\documentclass[a4paper,12pt]{article}
\usepackage{latexsym}
\usepackage{amsmath}
\usepackage{amssymb}
\usepackage{enumerate}

\usepackage{theorem}
\newtheorem{theorem}{Theorem}[section]

\newtheorem{lemma}[theorem]{Lemma}
\newtheorem{corollary}[theorem]{Corollary}

\newtheorem{TCT}{A New Type of Toponogov  Comparison Theorem}
\newtheorem{MT}{Main Theorem}
\newtheorem{ST}{Sector Theorem}
\newtheorem{AMOC}{Partial Answer to Milnor's Open Conjecture}
\newtheorem{MOC}{Milnor's Open Conjecture}
\newtheorem{IL}{Gromov's Isotopy Lemma}
\newtheorem{ML}{Model Lemma}

\theorembodyfont{\rmfamily}
\newtheorem{proof}{\textmd{\textit{Proof.}}}

\newtheorem{remark}[theorem]{Remark}
\newtheorem{example}[theorem]{Example}
\newtheorem{definition}[theorem]{Definition}

\newtheorem{acknowledgement}{\textmd{\textit{Acknowledgements.}}}

\makeatletter

\@addtoreset{equation}{section}
\makeatother

\newcommand{\qedd}{\hfill \Box}
\newcommand{\ve}{\varepsilon}
\newcommand{\lra}{\longrightarrow}
\newcommand{\wt}{\widetilde}
\newcommand{\wh}{\widehat}
\newcommand{\ol}{\overline}

\newcommand{\N}{\ensuremath{\mathbb{N}}}
\newcommand{\R}{\ensuremath{\mathbb{R}}}
\newcommand{\Sph}{\ensuremath{\mathbb{S}}}

\newcommand{\cD}{\ensuremath{\mathcal{D}}}

\newcommand{\cL}{\ensuremath{\mathcal{L}}}

\newcommand{\cN}{\ensuremath{\mathcal{N}}}
\newcommand{\cO}{\ensuremath{\mathcal{O}}}

\newcommand{\cU}{\ensuremath{\mathcal{U}}}

\def\Cut{\mathop{\mathrm{Cut}}\nolimits}

\title
{
Total Curvatures of Model Surfaces Control\\ 
Topology\hspace{-0.5mm} of Complete Open Manifolds with\\ 
Radial Curvature Bounded Below.\,II\footnote{
Mathematics Subject Classification (2000)\,:\,53C20, 53C21.}
\footnote{
Keywords\,:\,Toponogov's comparison theorem, 
cut locus, 
geodesic, 
radial curvature, 
total curvature}
}
\author{Kei KONDO $\cdot$ Minoru TANAKA}
\date{}
\begin{document}
\maketitle

\begin{abstract}
We prove, as our main theorem, 
the finiteness of topological type of 
a complete open Riemannian manifold $M$ with a base point $p \in M$ 
whose radial curvature at $p$ is bounded from below by that of a non-compact model 
surface of revolution $\wt{M}$ which 
admits a finite total curvature and has no pair of cut points in a sector. 
Here a sector is, by definition, a domain cut off by two meridians emanating 
from the base point $\tilde{p} \in \wt{M}$. 
Notice that our model $\wt{M}$ does not always satisfy the diameter growth condition introduced 
by Abresch and Gromoll. 
In order to prove the main theorem, 
we need a new type of the Toponogov  comparison theorem. 
As an application of the main theorem, 
we present a partial answer 
to Milnor's open conjecture on the fundamental group of complete open manifolds.
\end{abstract}

\section{Introduction}\label{sec:int}

Most comparison theorems which appear in differential geometry are originated from the 
Sturm comparison theorem (see \cite{St}, or \cite{Hr}). 
Bonnet \cite{Bo} would be the first researcher who applied the Sturm comparison theorem 
to differential geometry. 
In 1855, he proved that
the diameter of a compact surface does not exceed $\pi / \sqrt{\Lambda}$, 
if the Gaussian curvature of the surface is greater than a positive 
constant $\Lambda$.\par
In 1932, Schoenberg \cite{Sc} formulated the Sturm comparison theorem by 
using the sectional curvatures of Riemannian manifolds, 
and generalized Bonnet's theorem to any Riemannian manifolds. 
Therefore, it took {\em more than 70 years} to generalize Bonnet's theorem above, 
though von Mangoldt \cite{Ma} also applied the Sturm comparison theorem to 
differential geometry in 1881.\par 
In 1951, 
Rauch succeeded to compare the lengths of Jacobi fields 
along geodesics in differential Riemannian manifolds, 
and proved a well-known theorem, which is called the Rauch comparison theorem. 
Thus its special case is Schoenberg's theorem above. 
As an application of the Rauch comparison theorem, 
he proved a (topological) sphere theorem in \cite{R}. 
After this work, various kinds of sphere theorems have been 
proved (cf. \cite{K1}, \cite{K2}, \cite{B}, \cite{GS}, \cite{Sh2}, and so on).\par   
In 1959, Toponogov \cite{To1},\,\cite{To2} generalized the Rauch comparison theorem 
as a global version of the Rauch comparison theorem. 
He compared the angles of geodesic triangles in a Riemannian manifold 
and those in a complete $2$-dimensional Riemannian manifold 
of constant Gaussian curvature, 
which is well known as the Alexandrov--Toponogov comparison theorem 
(abbreviated to just the ``\,Toponogov\,'' comparison theorem). 
The Toponogov comparison theorem is now a very powerful tool for investigating global 
structures of Riemannian manifolds (cf. \cite{CG}, \cite{GS}, \cite{G1}, and so on).\par 
Some researchers tried to generalize the Toponogov comparison theorem. 
In 1985, Abresch \cite{A} generalized the Toponogov comparison theorem 
by a non-compact model surface of revolution with non-positive Gaussian curvature. 
He would be the first researcher who succeeded to generalize it in a rather satisfactory form, 
that is, in the {\em radial curvature geometry}.\par 

\bigskip

We will now introduce the radial curvature 
geometry for pointed complete open Riemannian manifolds\,: 
We call a complete open $2$-dimensional Riemannian manifold $(\wt{M}, \tilde{p})$ 
with a base point $\tilde{p} \in \wt{M}$ a 
{\em non-compact model surface of revolution} 
if its Riemannian metric $d\tilde{s}^2$ is expressed 
in terms of geodesic polar coordinates around $\tilde{p}$ as 
\begin{equation}\label{polar}
d\tilde{s}^2 = dt^2 + f(t)^2d \theta^2, \quad (t,\theta) \in (0,\infty) \times {\Sph_{\tilde{p}}^1}_. 
\end{equation}
Here, $f : (0, \infty) \lra \R$ is a positive smooth function which is extensible to a smooth 
odd function around $0$, 
and $\Sph^{1}_{\tilde{p}} := \{ v \in T_{\tilde{p}} \wt{M} \ | \ \| v \| = 1 \}$. 
The function $G \circ \tilde{\gamma} : [0,\infty) \lra \R$ is called the 
{\em radial curvature function} of $(\wt{M}, \tilde{p})$, 
where we denote by $G$ the Gaussian curvature of $\wt{M}$, 
and by $\tilde{\gamma}$ any meridian emanating from $\tilde{p} = \tilde{\gamma} (0)$. 
Remark that $f$ satisfies the differential equation 
\[
f''(t) + G (\tilde{\gamma}(t)) f(t) = 0
\]
with initial conditions $f(0) = 0$ and $f'(0) = 1$. 
The $n$-dimensional model surfaces of revolution are defined similarly, 
and they are completely classified in \cite{KK}.\par 
Let $(M,p)$ be a complete open $n$-dimensional Riemannian manifold with a base point $p \in M$.
We say that $(M, p)$ has 
{\em 
radial curvature at the base point $p$ bounded 
from below by that of 
a non-compact model surface of revolution $(\wt{M}, \tilde{p})$} 
if, along every unit speed minimal geodesic $\gamma: [0,a) \lra M$ 
emanating from $p = \gamma (0)$, 
its sectional curvature $K_M$ satisfies
\[
K_M(\sigma_{t}) \ge G (\tilde{\gamma}(t))
\]
for all $t \in [0, a)$ and all $2$-dimensional linear space $\sigma_{t}$ spanned by $\gamma'(t)$ 
and a tangent vector to $M$ at $\gamma(t)$.\par 
For example, if the Riemannian metric of $\wt{M}$ is 
$dt^2 + t^{2}d \theta^2$, or $dt^2 + \sinh^{2} t\,d \theta^2$, then 
$G (\tilde{\gamma}(t)) = 0$, or $G (\tilde{\gamma}(t)) = -1$, respectively. 
Furthermore, the radial curvature may change signs. 
Moreover, we can employ a model surface of revolution, as a comparison surface, 
satisfying $\lim_{t \to \infty} G (\tilde{\gamma}(t)) = -\infty$. 
Thus, {\em it is very natural as generalization of conventional comparison geometry 
to make use of a model surface of revolution instead of a complete simply connected surface 
of constant Gaussian curvature}. 

\bigskip

In 2003, the Toponogov comparison theorem was generalized 
by Itokawa, Machigashira, and Shiohama, by using a {\em 
von Mangoldt surface of revolution} as a $(\wt{M}, \tilde{p})$, 
in a very satisfactory form, i.e., their theorem contains 
the original Toponogov comparison theorem as a corollary (see \cite[Theorem 1.3]{IMS}). 
Here, a von Mangoldt surface of revolution is, 
by definition, a model surface of revolution 
whose radial curvature function is non-increasing on $[0, \infty)$. 
Paraboloids and $2$-sheeted hyperboloids are typical examples of a von Mangoldt surface of revolution. An untypical example of a von Mangoldt surface of revolution is found 
in \cite[Example 1.2]{KT1}, where its radial curvature function $G (\tilde{\gamma}(t))$ changes 
signs on $[0, \infty)$. 
We refer to \cite{T1} for other examples of a von Mangoldt surface of revolution. 
Thus, a von Mangoldt surface of revolution is a {\em very common} model 
as a reference space. 
We refer \cite{KO} and \cite{KT1} for applications of 
Itokawa--Machigashira--Shiohama' comparison theorem.

\bigskip

Poincar\'e \cite{Po} first introduced the notion of the cut locus for surfaces in 1905. 
He claimed that {\em the endpoints of the cut locus are cusps of the conjugate locus 
turned to the starting point}. 
We have explicitly determined such structures of 
model surfaces of revolution 
whose Gaussian curvature is monotonic along a subarc of a meridian 
(see \cite{T1}, \cite{GMST}, \cite{SiT1}, and \cite{SiT2}). 
For example, the cut locus $\Cut(\tilde{z})$ to each point
$\tilde{z} \in \wt{M} \setminus \{ \tilde{p} \}$
of a non-compact von Mangoldt surface of revolution $(\wt{M}, \tilde{p})$ is either an empty set, 
or a ray properly contained in the meridian $\theta^{-1} (\theta (\tilde{z}) + \pi)$ laying opposite to $\tilde{z}$, and that the endpoint of $\Cut(\tilde{z})$ is the first conjugate point to $\tilde{z}$ 
along the minimal geodesic from $\tilde{z}$ sitting in 
$ \theta^{-1} (\theta (\tilde{z})) \cup \theta^{-1} (\theta (\tilde{z}) + \pi)$ (see \cite[Main Theorem]{T1}). 
In particular, {\em any non-compact von Mangoldt surface of revolution has no pair of cut points 
in the sector $\wt{V} (\pi)$}, 
where we define a sector 
\[ 
\wt{V} (\delta) : = \left\{ \tilde{x} \in \wt{M} \, | \, 0 < \theta(\tilde{x}) < \delta \right\}
\]
for each constant number $\delta >0$.\par
However, we emphasize that {\em the cut locus on  a model 
surface of revolution need not be connected}. 
A model surface of revolution with a disconnected cut locus has been constructed 
in \cite[Section 2]{T2}. 
We also note that Gluck and Singer \cite{GlSi} constructed a smooth, 
but non-analytic surface of revolution embedded in $\R^{3}$ such that 
the cut locus of a point on the surface admits a branch point with infinite order, 
even one of strictly positive Gaussian curvature, 
so that its cut locus has also infinitely many endpoints.\par
Thus, it is not difficult to establish the Toponogov comparison theorem 
for a model surface of revolution admitting 
a very simple structure of the cut locus at each point as seen in \cite[Theorem 1.3]{IMS}. 
However, it is very difficult to establish the Toponogov comparison theorem 
for an arbitrary model surface of revolution. 
The cause of the difficulty lies in the complexity of cut loci of model 
surfaces of revolution as stated above. 
Therefore, {\em the structure of cut loci
of model surfaces of revolution allows a generalization of the Toponogov
comparison theorem in radial curvature geometry}.

\bigskip

In this article, we are concerned with a generalization of 
the Toponogov comparison theorem to the radial curvature geometry. 
That is, we need a new type of the Toponogov  comparison theorem 
in order to prove our main theorem\,:

\begin{TCT} \hspace{-3.2mm} {\rm (Theorem \ref{thm4.8} in Section \ref{sec:TCT})} \par
Let $(M,p)$ be a complete open Riemannian $n$-manifold $M$ 
whose radial curvature at the base point $p$ is bounded from below by
that of a non-compact model surface of revolution $(\wt{M}, \tilde{p})$. 
Assume that $\wt{M}$ admits a sector $\wt{V}(\delta_{0})$ 
for some $\delta_{0} \in (0, \pi)$ which has no pair of cut points. 
Then, for every geodesic triangle $\triangle(pxy)$ in $M$ with $\angle (xpy) < \delta_{0}$,
there exists a geodesic triangle 
$\wt{\triangle} (pxy) :=\triangle(\tilde{p}\tilde{x}\tilde{y})$ in $\wt{V}(\delta_{0})$ such that
\begin{equation}\label{TCT-length}
d(\tilde{p},\tilde{x})=d(p,x), \quad d(\tilde{p},\tilde{y})=d(p,y), \quad d(\tilde{x},\tilde{y})=d(x,y) 
\end{equation}
and that
\begin{equation}\label{TCT-angle}
\angle (xpy) \ge \angle (\tilde{x}\tilde{p}\tilde{y}), \quad  
\angle (pxy) \ge \angle (\tilde{p}\tilde{x}\tilde{y}), \quad
\angle (pyx) \ge \angle (\tilde{p}\tilde{y}\tilde{x}). 
\end{equation}
Here we denote by $d(\, \cdot \,, \, \cdot \, )$ the distance function 
induced from the Riemannian structure of $M$, or $\wt{M}$, 
and by $\angle(pxy)$ the angle between the minimal geodesics from $x$ to $p$ and $y$ 
forming the triangle $\triangle(pxy)$.
\end{TCT}

\medskip\noindent
The assumption on $\wt{V}(\delta_{0})$ in our comparison theorem 
is automatically satisfied, if we employ a von Mangoldt surface of revolution, 
or a Cartan--Hadamard surface of revolution (i.e., $G$ is non-positive on $\wt{M}$) 
as a $(\wt{M}, \tilde{p})$. 
Therefore, our comparison theorem contains 
Itokawa--Machigashira--Shiohama' comparison theorem above as a corollary. 
The proof of our comparison theorem has completely different techniques from used 
in \cite[Chapter 2]{CE}, \cite[Chapter IV]{Sa}, \cite{A}, 
and \cite{IMS} (see Section \ref{sec:TCT} for the proof). 

\bigskip

Before stating our main theorem, we will mention a few related results for our theorem
to clarify our aim\,: Abresch and Gromoll \cite{AG} proved the finiteness of topological type of 
a complete open $n$-dimensional Riemannian manifold $X$ with non-negative Ricci curvature 
outside the open distance $t_{0}$-ball around $p \in X$ for some constant $t_{0} > 0$, 
{\em however, admitting sectional curvature bounded from below 
by some negative constant everywhere, 
and furthermore admitting diameter growth of small order $o(t^{1/ n})$ for some $p \in X$}. 
Here,\,``$X$ has finite topological type.''\,means that 
$X$ is homeomorphic to the interior of a compact manifold with boundary. 
Although their result and technique have influenced many articles (cf.\,Sormani's study \cite{So}), 
it looks very {\em restrictive to assume that a complete open 
Riemannian $n$-manifold admits diameter growth of small order $o(t^{1/ n})$}, 
if once you live in the world of the radial curvature geometry. 
The next example shows that the diameter growth condition is too restrictive.

\begin{example}\label{exa1.1}
Let $(\wt{M}, \tilde{p})$ be a non-negatively curved non-compact model surface of revolution with 
its metric (\ref{polar}), and admitting diameter growth $\cD(t, \tilde{p})=o(t^{1/ 2})$. 
We denote by $\cL(t, \tilde{p})$ the length of boundary of the open distance $t$-ball around $\tilde{p}$. 
Since $\cD(t, \tilde{p}) = o(t^{1/ 2})$ and 
\[
\lim_{t \to \infty} \frac{\cD(t, \tilde{p})}{\cL(t, \tilde{p})}
\] 
is positive by \cite[Lemma 7.3.3]{SST}, 
there exist positive numbers $\ve_{1}$ and $\ve_{2}$ such that 
\begin{equation}\label{exa1.1-1}
{
2 \pi f(t) = \cL(t, \tilde{p}) = \frac{\cL(t, \tilde{p})}{\cD(t, \tilde{p})} \cdot \cD(t, \tilde{p}) \le \ve_{1} \cD(t, \tilde{p}) \le \ve_{2} \sqrt{t}
}_. 
\end{equation}
By (\ref{exa1.1-1}), $\wt{M}$ satisfies
\begin{equation}\label{exa1.1-2}
{
\int_{1}^{\infty} \frac{1}{f(t)^2}\,dt = \infty
}_.
\end{equation}
Let $(M, p)$ be a complete open $n$-dimensional Riemannian 
manifold $M$ whose radial curvature at 
the base point $p$ is bounded from below by that of $\wt{M}$. 
Then (\ref{exa1.1-2}) controls the topology of $(M, p)$, that is, 
it follows from \cite[Theorem 1.2]{ST2} that 
{\em $(M, p)$ is isometric to the $n$-dimensional model $\wt{M}^{n}$ satisfying $(\ref{exa1.1-2})$}. 
\end{example}

\medskip

Authors \cite{KT1} have recently reached stronger conclusion 
than Abresch and Gromoll' result above, 
in which diameter growth condition is replaced by an assumption 
on total curvatures of model surfaces\,: 

\begin{theorem}{\rm (\cite[Corollary to Main Theorem]{KT1})}\label{KT}
Let $(M, p)$ be a complete open Riemannian $n$-manifold $M$ whose radial curvature at 
the base point $p \in M$ is bounded from below by that of a non-compact von Mangoldt surface of 
revolution $(\wt{M}, \tilde{p})$. 
If $\wt{M}$ admits $c(\wt{M}) > \pi$, 
then $M$ has finite topological type, and the isometry group of $M$ is compact. 
Here $c(\wt{M})$ denotes the total curvature of $\wt{M}$.
\end{theorem}

\medskip\noindent
The assumption $c(\wt{M}) > \pi$, of which one may complain being too much strong, 
has been assumed in order to generalize Shiohama's result (\cite[Main Theorem]{Sh1}) 
in geometry of complete open surfaces 
to any dimensional complete open Riemannian manifolds (see \cite[Main Theorem]{KT1}). 
Therefore, our main purpose of this article is to show 
the finiteness of topological type of a complete open Riemannian manifold 
with a wider class of metrics than those described in Theorem \ref{KT}. 

\bigskip

Our main theorem is the following\,:

\medskip

\begin{MT} \hspace{-3.2mm} {\rm (Theorem \ref{thm2.1} in Section \ref{PMT})} \par 
Let $(M,p)$ be a complete open Riemannian $n$-manifold $M$ 
whose radial curvature at the base point $p$ is bounded from below by
that of a non-compact model surface of revolution $(\wt{M}, \tilde{p})$. 
If $\wt{M}$ admits $c(\wt{M}) > - \infty$ and has no pair of cut points 
in $\wt{V} (\delta_{0})$ for some $\delta_{0} \in (0, \pi)$, 
then $M$ has finite topological type. 
\end{MT}

\bigskip\noindent
Remark that $\wt{M}$ admits $c(\wt{M}) > - \infty$ if and only if  
\[
\int_{\wt{M}} |G|\,d\wt{M} < \infty
\]
(see \cite[Definition 2.1.3]{SST} and the paragraph after the definition).\par 
In Main Theorem, 
the assumption on the existence of $\wt{V} (\delta_{0})$ is necessary 
in order to establish a new type of Toponogov  comparison theorem 
as stated above. 
Notice that the model surface of revolution $\wt{M}$ in our 
Main Theorem cannot always be replaced by a Cartan--Hadamard surface of revolution with 
finite total curvature which bounds the radial curvature of $M$ from below.\par
Next, we will mention significance of finite total curvature in Main Theorem. 
Consider a non-compact model surface of revolution $(\wt{M}, \tilde{p})$ 
which satisfies one of the following conditions for some constant $R_{0} > 0$\,: 

\begin{enumerate}[{\bf (VM)}]
\item
The radial curvature function is {\em non-increasing on} $[R_{0}, \infty)$.
\end{enumerate} 

\begin{enumerate}[{\bf (CH)}]
\item
The radial curvature function is {\em non-positive on} $[R_{0}, \infty)$.
\end{enumerate}

\medskip

Then, the following theorem clarifies real significance 
of finite total curvature and also the validity of our Main Theorem\,:

\medskip

\begin{ST} \hspace{-3.2mm} {\rm (Theorem \ref{thm3.6} in Section \ref{sec:cut})} \par 
Let $(\wt{M}, \tilde{p})$ be a non-compact model surface of revolution satisfying the (VM), 
or the (CH) for some $R_{0} > 0$. 
If $\wt{M}$ admits $c(\wt{M}) > - \infty$, then there exists a number $\delta_{0} \in (0, \pi)$ 
such that there is no pair of cut points in $\wt{V} ( \delta_{0} )$. 
\end{ST}

\medskip\noindent
From the radial curvature geometry's standpoint, 
we therefore feel that our Main Theorem has more natural assumption than that on 
Abresch and Gromoll' result above. 
Another related result for Main Theorem is Sormani's \cite{So}. 
She has proved that if a complete open $n$-dimensional Riemannian manifold 
with non-negative Ricci curvature everywhere admits some diameter growth condition, 
then the manifold has a finitely generated fundamental group. 
However, by the same reason above, 
the diameter growth condition on her result also looks very restrictive. 

\bigskip

As an application of Main Theorem and Sector Theorem, 
we will present a partial answer to Milnor's open conjecture. 
The conjecture is stated as follows\,:

\begin{MOC} \hspace{-1.5mm}{\rm (see the line right after \cite[Theorem 1]{Mi})} \par
A complete open Riemannian manifold with non-negative Ricci curvature everywhere 
must have a finitely generated fundamental group. 
\end{MOC}

Then, the assumption on the existence of a non-compact model surface of revolution is {\em natural}. 
In fact, any complete open $n$-dimensional Riemannian manifold has a non-compact 
model surface of revolution\,:

\begin{ML} \hspace{-3.2mm} {\rm (Lemma \ref{lem5.1} in Section \ref{sec:MOC})} \par
Let $M$ be a complete open Riemannian $n$-manifold 
and $p \in M$ {\bf any} fixed point. 
Then, there exists a locally Lipschitz function $G(t)$ on $[0, \infty)$ 
such that the radial curvature of $M$ at $p$ is bounded from below by 
that of the non-compact model surface of revolution with radial curvature function $G(t)$. 
\end{ML}

\medskip\noindent
Now, let $G(t)$ be the Lipschitz function in Model Lemma, and set 
\[
G^{*}(t) := \min \left\{ 0, G(t) \right\}_.
\]
Consider a non-compact model surface of revolution $(M^{*}, p^{*})$ 
with its metric 
\begin{equation}\label{1.1}
g^{*} = dt^2 +  m(t)^2d \theta^2, \quad (t,\theta) \in (0,\infty) \times \Sph_{p^{*}}^1
\end{equation}
satisfying the differential equation
\[
m''(t) + G^{*} (t) m(t) = 0
\]
with initial conditions $m(0) = 0$ and $m'(0) = 1$. 
Notice that the metric (\ref{1.1}) is {\em not always differentiable around} 
the base point $p^{*} \in M^{*}$. 
Then, it follows from Model Lemma, Sector Theorem, and Main Theorem that 
we have the following partial answer to Milnor's open conjecture\,:

\medskip

\begin{AMOC} \hspace{-1.5mm}{\rm (Theorem \ref{thm5.3} in Section \ref{sec:MOC})} \par
Let $M$ be a complete open Riemannian $n$-manifold, $p \in M$ {\bf any} fixed point, 
and $(M^{*}, p^{*})$ a comparison model surface of revolution, constructed from $(M, p)$, 
with its metric (\ref{1.1}). 
If $G^{*} (t)$ satisfies 
\[
\int^{\infty}_{0} \left( - t \cdot G^{*} (t) \right) dt < \infty, 
\]
then the total curvature $c(M^{*})$ is finite. 
In particular, then $M$ has a finitely generated fundamental group. 
\end{AMOC}

\bigskip

Another approach to Milnor's open conjecture has been done by Wilking \cite{W}. 
He tried to prove the conjecture for manifolds with abelian fundamental groups, i.e., he proved that\,: 
Let $X$ be a complete $n$-dimensional Riemannian manifold $X$ 
with non-negative Ricci curvature everywhere, and $\wh{X}$ the universal covering space of $X$. 
Then, if $I(\wh{X}) / I_{0}(\wh{X})$ is finitely generated, 
the fundamental group of $X$ is too. 
Here, $I_{0}(\wh{X})$ denotes the identity component 
of the isometry group $I(\wh{X})$ of $\wh{X}$. 

\bigskip

The organization of this article is as follows. 
In Section \ref{PMT}, this article reaches the climax, that is, 
we prove Main Theorem (Theorem \ref{thm2.1}) 
by applying some results in Sections \ref{sec:cut} and \ref{sec:TCT}. 
In Section \ref{sec:cut}, we investigate the relationship between 
a non-compact model surface of revolution $(\wt{M}, \tilde{p})$ 
with its finite total curvature and its cut locus. 
As a main theorem in Section \ref{sec:cut}, 
we finally prove Sector Theorem (Theorem \ref{thm3.6}) using results in Section \ref{sec:cut}. 
In Section \ref{sec:TCT}, 
we establish a new type of 
the Toponogov comparison theorem (Theorem \ref{thm4.8}). 
The key tools of the proof are our Alexandrov convexity (Lemma \ref{lem4.3}) 
and Essential Lemma (Lemma \ref{lem4.6}). 
In Section \ref{sec:MOC}, we prove Model Lemma (Lemma \ref{lem5.1}). 
After recalling some differential inequality (Lemma \ref{DIL}),  
we prove a partial answer (Theorem \ref{thm5.3}) to Milnor's open conjecture, 
using Model Lemma, the differential inequality, Sector Theorem, and Main Theorem.
Finally, we prove a corollary (Corollary \ref{cor5.4}) 
to Main Theorem, which is another partial answer to Milnor's open conjecture.

\bigskip

In the following sections, 
all geodesics will be normalized, unless otherwise stated. 

\begin{acknowledgement}
The first named author would like to express to Professor Karsten Grove 
my deepest gratitude for his encouragement, and his interest in our work. 
We are very grateful to Professors Takashi Shioya and Shin-ichi Ohta 
for their helpful comments on the first draft of this article. 
Finally, we also thank the referee for careful reading of the manuscript, valuable suggestions, 
and helpful comments on the manuscript, which, no doubt, have improved the presentation of 
this article.
\end{acknowledgement}

\section{Proof of Main Theorem}\label{PMT}

In 1977, Grove and Shiohama \cite{GS} first introduced the notion of critical point 
of distance functions to prove their famous theorem, which is called the diameter sphere theorem. 
The definition of the critical point of distance functions is given as follows\,:

\begin{definition}
Let $M$ be a complete Riemannian manifold. 
For any fixed point $p \in M$, 
a point $q \in M \setminus \{ p \}$ is called 
a {\it critical point of} $d(p, \, \cdot \, )$ (or {\em critical point for} $p$) if, 
for every nonzero tangent vector $v \in T_{q} M$, 
we find a minimal geodesic $\gamma$ emanating from $q$ to $p$
satisfying 
\[
{\angle(v, \gamma'(0)) \le \frac{\pi}{2}}_.
\]
Here, we denote by $\angle(v, \gamma'(0))$ 
the angle between two vectors $v$ and $\gamma'(0)$ in $T_{q} M$.
\end{definition}
In 1981, Gromov \cite{G1} refined the technique given by Grove and Shiohama in \cite{GS} 
in order to estimate an upper bound 
on the sum of the Betti numbers over any fixed field for compact (connected) 
Riemannian manifolds with non-negative sectional curvature everywhere\,:

\begin{IL}
Let $M$ be a complete Riemannian manifold. 
If $0 < R_1 < R_2 \le \infty$, 
and if $\overline{B_{R_2}(p)} \setminus B_{R_1}(p)$ has no critical point for $p \in M$, 
then $\overline{B_{R_2}(p)} \setminus B_{R_1}(p)$ is homeomorphic to 
$\partial B_{R_1}(p) \times [R_1, R_2 ]$. 
Here, $B_{R_{i}}(p)$, $i =1, 2$, 
are the open distance $R_{i}$-balls around $p$.
\end{IL}

\medskip

From now on, 
let $(\wt{M}, \tilde{p})$ be a non-compact model surface of revolution, 
and let $(M,p)$ be a complete open Riemannian $n$-manifold $M$ 
whose radial curvature at the base point $p \in M$ is bounded from below by
that of the $(\wt{M}, \tilde{p})$. 
Admitting Lemmas \ref{lem3.3}, \ref{lem4.6}, and Theorem \ref{thm4.8} in Sections $3$ and $4$, 
we can prove that\,:

\begin{theorem}\label{thm2.1}
If $\wt{M}$ admits $c(\wt{M}) > - \infty$ and has no pair of cut points 
in $\wt{V} (\delta_{0})$ for some $\delta_{0} \in (0, \pi)$, 
then $M$ has finite topological type. 
\end{theorem}

\begin{proof}
From the Isotopy Lemma, it is sufficient to prove that the set of critical points of 
$d(p, \, \cdot \, )$ is bounded. 
Suppose that there exists a divergent sequence $\{ q_{i} \}$ of critical points $q_{i} \in M$
of $d(p, \, \cdot \, )$. 
Let $\gamma_{i} : [0, d(p, q_{i})] \lra M$ be 
a minimal geodesic segment emanating from $p$ to $q_{i}$. 
We may assume, by taking a subsequence of $\{ q_{i} \}$, if necessary, that 
the sequence $\{ \gamma_{i} \}$ converges to a ray $\gamma$ emanating from $p$.
Thus, there exists a sufficiently large $i_{0} \in \N$ such that 
\begin{equation}\label{thm2.1-1}
\angle (\gamma'(0), \gamma'_{i}(0)) < \delta_{0}
\end{equation}
for any $i \ge i_{0}$. 
From now on, choose any $i \ge i_{0}$ and fix it. 
Let $\ve$ be an arbitrary positive number less than $\pi / 2$. 
Applying the Cohn\,-\,Vossen's technique (see \cite{CV}, or \cite[Lemma 2.2.1]{SST}), 
we can choose a positive number $t_{i}$ satisfying 
\begin{equation}\label{thm2.1-2}
\angle (\gamma'(t_{i}), \eta'_{i}(s_{i})) < \ve.
\end{equation}
Here, $\eta_{i} : [0, s_{i}] \lra M$ denotes a minimal geodesic segment emanating from $q_{i}$ 
to $\gamma(t_{i})$. 
Since the geodesic triangle $\triangle (p q_{i} \gamma(t_{i})) \subset M$, 
consisting of the edges $\gamma_{i}$, $\eta_{i}$, and $\gamma|_{[0, \, t_{i}]}$, 
satisfies 
\[
\angle (q_{i} p \gamma(t_{i})) < \delta_{0}
\] 
by (\ref{thm2.1-1}), 
it follows from Theorem \ref{thm4.8} and (\ref{thm2.1-2}) that 
there exists a geodesic triangle 
$
\wt{\triangle} (p q_{i} \gamma(t_{i})) 
:= \triangle(\tilde{p}\tilde{q}_{i}\tilde{\gamma}(t_{i})) \subset \wt{M}
$ 
satisfying (\ref{thm4.8-length}) (for $x = q_{i}$ and $y = \gamma (t_{i})$), 
\[
\angle (q_{i} p \gamma(t_{i})) \ge \angle (\tilde{q}_{i}\tilde{p}\tilde{\gamma}(t_{i}))
\]
and 
\begin{equation}\label{thm2.1-3.5}
\angle (\tilde{q}_{i} \tilde{\gamma}(t_{i}) \tilde{p}) < \ve.
\end{equation}
Since  
\[
{\lim_{i \to \infty} \angle (q_{i} p \gamma(t_{i})) = 0}_,
\]
we have 
\begin{equation}\label{thm2.1-3}
{\lim_{i \to \infty} \angle (\tilde{q}_{i}\tilde{p}\tilde{\gamma}(t_{i})) = 0}_.
\end{equation}
On the other hand, since each $q_{i}$ is a critical point of $d(p, \, \cdot \, )$, 
there exists a minimal geodesic segment $\sigma_{i} : [0, d(p, q_{i})] \lra M$ emanating from 
$q_{i}$ to $p$ such that 
\begin{equation}\label{thm2.1-4}
{\angle (\sigma'_{i}(0),  \eta'_{i}(0)) \le \frac{\pi}{2}}_.
\end{equation}
Let $\triangle (p \sigma_{i} (0) \gamma(t_{i})) \subset M$ denote 
the geodesic triangle consisting of the edges $\sigma_{i}$, $\eta_{i}$, and $\gamma|_{[0, \, t_{i}]}$. 
Since $\triangle (p \sigma_{i} (0) \gamma(t_{i}))$ has the same side lengths 
as $\triangle (p q_{i} \gamma(t_{i}))$, 
the triangle $\triangle (p \sigma_{i} (0) \gamma(t_{i}))$ admits the triangle 
$\wt{\triangle} (p q_{i} \gamma(t_{i})) \subset \wt{V}(\delta_{0})$ 
satisfying (\ref{lem4.6-length}) (for $x = \sigma_{i} (0)$ and $y = \gamma (t_{i})$) 
in Lemma \ref{lem4.6}. 
Thus, by Lemma \ref{lem4.6}, we have 
\begin{equation}\label{thm2.1-5}
\angle (\sigma'_{i}(0),  \eta'_{i}(0)) \ge \angle (\tilde{p} \tilde{q}_{i} \tilde{\gamma}(t_{i})).
\end{equation}
By (\ref{thm2.1-4}) and (\ref{thm2.1-5}), 
we get 
\begin{equation}\label{thm2.1-7}
{\angle (\tilde{p} \tilde{q}_{i} \tilde{\gamma}(t_{i})) \le \frac{\pi}{2}}_.
\end{equation}
Applying the Gauss\,--\,Bonnet Theorem to the geodesic triangle $\wt{\triangle} (p q_{i} \gamma(t_{i}))$, 
we have 
\begin{equation}\label{thm2.1-8}
{\int_{\wt{\triangle} (p q_{i} \gamma(t_{i}))} G\, d\wt{M} 
= \angle (\tilde{q}_{i} \tilde{p} \tilde{\gamma}(t_{i})) 
+ \angle (\tilde{p} \tilde{q}_{i} \tilde{\gamma}(t_{i})) 
+ \angle (\tilde{q}_{i} \tilde{\gamma}(t_{i}) \tilde{p}) - \pi}_.
\end{equation}
Since 
\[
\lim_{i \to \infty} \int_{\wt{\triangle} (p q_{i} \gamma(t_{i}))} G\, d\wt{M} = 0
\]
by Lemma \ref{lem3.3}, 
we have, by (\ref{thm2.1-8}), 
\begin{equation}\label{thm2.1-9}
\lim_{i \to \infty} 
\left( 
\angle (\tilde{q}_{i} \tilde{p} \tilde{\gamma}(t_{i})) 
+ \angle (\tilde{p} \tilde{q}_{i} \tilde{\gamma}(t_{i})) 
+ \angle (\tilde{q}_{i} \tilde{\gamma}(t_{i}) \tilde{p})
\right) 
= \pi.
\end{equation}
By (\ref{thm2.1-3.5}), (\ref{thm2.1-3}), (\ref{thm2.1-7}), and (\ref{thm2.1-9}), 
we get 
\[
{\frac{\pi}{2} \le \ve}_.
\]
This is a contradiction. 
Therefore, the set of critical points of $d(p, \, \cdot \, )$ is bounded. 
$\qedd$
\end{proof}

\section{Sector Theorem and Cut Loci of Model Surfaces}\label{sec:cut}

We will first 
introduce some fundamental tools in geometry of surfaces of revolution. 
For details of  the geometry, readers can refer 
to \cite[Chapter 7]{SST} (also refer to \cite{T1}, \cite{GMST}, and \cite{SiT2}). 
Remark that the following results hold for all non-compact model surfaces of revolution 
except for Theorem \ref{thm3.6}.

\bigskip

For a non-compact model surface of revolution $(\wt{M}, \tilde{p})$ 
whose metric satisfies (\ref{polar}), 
a unit speed geodesic $\tilde{\sigma} : [0, a) \lra \wt{M}$ $(0 < a \le \infty)$ is expressed by 
\[
\tilde{\sigma}(s) = (t(\tilde{\sigma}(s)), \theta(\tilde{\sigma}(s))) =: (t(s), \theta(s)).
\]
Then, there exists a non-negative constant $\nu$ depending only on $\tilde{\sigma}$ such that 
\begin{equation}\label{clairaut}
\nu = f(t(s))^{2} |\theta' (s)| = f(t(s)) \sin \angle(\tilde{\sigma}'(s), (\partial / \partial t)_{\tilde{\sigma}(s)}).
\end{equation}
This (\ref{clairaut}) is a famous formula, which is called the {\em Clairaut relation}. 
The constant $\nu$ is called the {\em Clairaut constant of} $\tilde{\sigma}$. 
Remark that, by (\ref{clairaut}),  
\begin{center}
{\em 
$\nu > 0$ if and only if $\tilde{\sigma}$ is not a meridian, or its subarc.
} 
\end{center} 
Since $\tilde{\sigma}$ is unit speed, we have, by (\ref{clairaut}), 
\begin{equation}\label{geodesic-eq}
{t'(s) = \pm \frac{\sqrt{f(t(s))^{2} - \nu^2}}{f(t(s))}}_.
\end{equation}
Remark that, by (\ref{geodesic-eq}),
\begin{center}
{\em
$t'(s) = 0$ if and only if $f(t(s)) = \nu$.
}
\end{center}
It follows from (\ref{clairaut}) and (\ref{geodesic-eq}) that, 
for a unit speed geodesic $\tilde{\sigma}(s) = (t(s), \theta(s))$, $s_1 \le s \le s_2$, with 
the Clairaut constant $\nu$, 
\begin{equation}\label{angle}
\theta(s_2) - \theta(s_1) 
= \lambda(t'(s)) \int_{t(s_1)}^{t(s_2)} \frac{\nu}{f(t) \sqrt{f(t)^{2} -\nu^2}}\,dt
\end{equation}
holds if $t'(s) \not= 0$ on $(s_1, s_2)$. 
Here, $\lambda(t'(s))$ denotes the sign of $t'(s)$.

\medskip

\begin{lemma}\label{lem3.2}
Let $(\wt{M}, \tilde{p})$ be a non-compact model surface of revolution, 
and $\wt{V}_{i}$ denote $\wt{V} ( 1 / i )$ for each $i \in \N$. 
Assume that there exist a constant $t_{0} > 0$ and a sequence 
\[
\left\{\tilde{\sigma}_{i} : [0, \ell_{i}] \lra \wt{V}_{i} \right\}_{i \in \N} 
\]
of unit speed geodesic segments such that 
\begin{equation}\label{lem3.2-1}
\tilde{\sigma}_{i}([0, \ell_{i}]) \cap \ol{B_{t_{0}} (\tilde{p})} \not= \emptyset
\end{equation}
for each $i \in \N$, and that 
\begin{equation}\label{lem3.2-2}
{\liminf_{i \to \infty} t(\tilde{\sigma}_{i}(\ell_{i})) > t_{0}}_.
\end{equation}
Then, 
\[
{\lim_{i \to \infty} \nu_{i} = 0}_.
\]
holds. 
Here, $\nu_{i}$ denotes the Clairaut constant of $\tilde{\sigma}_{i}$. 
\end{lemma}

\begin{proof}
By the assumptions (\ref{lem3.2-1}) and (\ref{lem3.2-2}), we can find a constant number 
\[
u_{i} := \sup \{ s \in (0, \ell_{i}) \, | \, t(\tilde{\sigma}_{i}(s)) = t_{0} \}
\]
for each $i \in \N$. 
Moreover, it follows from (\ref{lem3.2-2}) that 
for any sufficiently large $i \in \N$,  
there exists a constant number $t_{1} > 0$ such that 
\[
t_{0} < t_{1} < t(\tilde{\sigma}_{i}(\ell_{i})).
\]
Then, we can find a constant number 
$v_{i} \in (t \circ \tilde{\sigma}_{i})^{-1} (t_{1}) \subset (0, \ell_{i})$ 
such that $v_{i} > u_{i}$. 
Since $\tilde{\sigma}_{i}([u_{i}, v_{i}]) \subset \wt{V}_{i}$, 
we have, by (\ref{angle}),    
\begin{align}\label{lem3.2-3}
\frac{1}{i} 
&> 
\int_{t(\tilde{\sigma}_{i}(u_{i}))}^{t(\tilde{\sigma}_{i}(v_{i}))} 
\frac{\nu_{i}}{f(t) \sqrt{f(t)^{2} -\nu_{i}^2}}\,dt \notag\\[2mm]
&= 
\int_{t_{0}}^{t_{1}} \frac{\nu_{i}}{f(t) \sqrt{f(t)^{2} -\nu_{i}^2}}\,dt \notag\\[2mm]
&\ge 
\nu_{i} \cdot \int_{t_{0}}^{t_{1}} \frac{1}{f(t)^{2}}\,dt 
> 0
\end{align}
for any sufficiently large $i \in \N$. 
Thus, by (\ref{lem3.2-3}), 
we get 
\[
{\lim_{i \to \infty} \nu_{i} = 0}_.
\]
$\qedd$
\end{proof}

\medskip

\begin{lemma}\label{lem3.3}
Let $(\wt{M}, \tilde{p})$ be a non-compact model surface of revolution, 
and set 
\[
\ol{V}(\theta_{i}) := \left\{ \tilde{x} \in \wt{M} \, | \, 0 \le \theta(\tilde{x}) \le \theta_{i} \right\}
\]
for each $i \in \N$, where $\{ \theta_{i} \}$ is a sequence of positive numbers 
convergent to $0$. 
If $\wt{M}$ admits $c(\wt{M}) > - \infty$, then 
\[
\lim_{i \to \infty} \int_{\ol{V}(\theta_{i})} |G|\, d\wt{M} = 0
\]
holds. 
In particular, 
\[
\lim_{i \to \infty} \int_{\wt{\triangle}_{i}} G\, d\wt{M} = 0
\]
holds for any sequence $\{ \wt{\triangle}_{i} \}$ 
of geodesic triangles $\wt{\triangle}_{i} \subset \ol{V}(\theta_{i})$.
\end{lemma}

\begin{proof}
Since $c(\wt{M}) > - \infty$, 
for any $\ve > 0$, there exists a number $r(\ve) > 0$ such that 
\begin{equation}\label{lem3.3-1}
{\int_{\wt{M} \setminus B_{r(\ve)}(\tilde{p})} |G|\,d\wt{M} < \frac{\ve}{2}}_,
\end{equation}
where $B_{r(\ve)}(\tilde{p}) \subset \wt{M}$
is the open distance $r(\ve)$-ball around $\tilde{p} \in \wt{M}$. 
Then, there exists $i_{0}(\ve) \in \N$ such that 
\begin{equation}\label{lem3.3-2}
\int_{\ol{V}(\theta_{i}) \cap B_{r(\ve)}(\tilde{p})} |G|\,d\wt{M} 
= \frac{\theta_{i}}{2\pi}\int_{B_{r(\ve)}(\tilde{p})} |G|\,d\wt{M} < \frac{\ve}{2}
\end{equation}
holds for all $i > i_{0}(\ve)$.
Therefore, by (\ref{lem3.3-1}) and (\ref{lem3.3-2}), we get the first assertion, that is, 
\begin{align}\label{lem3.3-3}
\int_{\ol{V}(\theta_{i})} |G|\,d\wt{M} 
\le \int_{\ol{V}(\theta_{i}) \cap B_{r(\ve)}(\tilde{p})} |G|\,d\wt{M} + 
\int_{\wt{M} \setminus B_{r(\ve)}(\tilde{p})} |G|\,d\wt{M} < \ve
\end{align}
for all $i > i_{0}(\ve)$. 
Furthermore, by (\ref{lem3.3-3}), 
\[
\left| \int_{\wt{\triangle}_{i}} G\,d\wt{M} \,\right| \le \int_{\wt{\triangle}_{i}} |G|\,d\wt{M} 
\le \int_{\ol{V}(\theta_{i})} |G|\,d\wt{M} < \ve
\]
holds for all $i > i_{0}(\ve)$, which is the second assertion. 
$\qedd$
\end{proof}

\begin{lemma}\label{lem3.4}{\bf (Key Lemma)}
Let $(\wt{M}, \tilde{p})$ be a non-compact model surface of revolution. 
If $\wt{M}$ admits $c(\wt{M}) > - \infty$, then, for each $t > 0$, there exists a constant number 
$\delta(t) \in (0, \pi)$ such that
\[
\tilde{\sigma}([0, \ell]) \cap \ol{B_{t}(\tilde{p})} = \emptyset 
\]
holds for any minimal geodesic segment 
$\tilde{\sigma} : [0, \ell] \lra \wt{V} (\delta(t)) \subset \wt{M}$, 
along which $\tilde{\sigma}(0)$ is conjugate to $\tilde{\sigma}(\ell)$. 
\end{lemma}

\begin{proof}
Since $|\theta (\tilde{\sigma}(0)) - \theta (\tilde{\sigma}(\ell))| < \pi$ holds 
for all minimal geodesic segments $\tilde{\sigma} : [0, \ell] \lra \wt{M} \setminus \{ \tilde{p}\}$, 
it is sufficient to show a required number $\delta(t)$ is positive. 
Suppose that there exist a constant $t_{0} > 0$ and a sequence 
\[
\left\{\tilde{\sigma}_{i} : [0, \ell_{i}] \lra \wt{V}_{i} \right\}_{i \in \N}  
\]
of minimal geodesic segments, along which 
$\tilde{\sigma}_{i}(0)$ is conjugate to $\tilde{\sigma}_{i}(\ell_{i})$ for each $i \in \N$, 
such that 
\begin{equation}\label{lem3.4-2}
\tilde{\sigma}_{i}([0, \ell_{i}]) \cap \ol{B_{t_{0}} (\tilde{p})} \not= \emptyset
\end{equation}
for each $i \in \N$.
Here, we set $\wt{V}_{i} := \wt{V} ( 1 / i )$ for each $i \in \N$. 
Since 
$\tilde{\sigma}_{i}(0)$ is conjugate to $\tilde{\sigma}_{i}(\ell_{i})$ along $\tilde{\sigma}_{i}$ 
for each $i \in \N$, 
we see $\displaystyle{\lim_{i \to \infty} \ell_{i} = \infty}$, 
and hence we may assume 
\begin{equation}\label{lem3.4-3}
{\liminf_{i \to \infty} t(\tilde{\sigma}_{i}(\ell_{i})) > t_{0}}_.
\end{equation}
Thus, by Lemma \ref{lem3.2}, 
we have
\begin{equation}\label{lem3.4-4}
{\lim_{i \to \infty} \nu_{i} = 0}_,
\end{equation}
where $\nu_{i}$ is the Clairaut constant of $\tilde{\sigma}_{i}$. 
Since $\tilde{\sigma}_{i}(0)$ is conjugate to $\tilde{\sigma}_{i}(\ell_{i})$ along $\tilde{\sigma}_{i}$ 
for each $i \in \N$, 
there exists $a_{i} \in [0, \ell_{i}]$ such that 
\begin{equation}\label{lem3.4-5}
(t \circ \tilde{\sigma}_{i})' (a_{i}) = 0
\end{equation}
(cf. \cite[Proposition 7.2.1]{SST}). 
Let $u_{i} \in [0, \ell_{i}]$ be the parameter value of $\tilde{\sigma}_{i}$ such that 
\[
t(\tilde{\sigma}_{i}(u_{i})) = t_{0}.
\]
We consider a geodesic triangle 
$\wt{\triangle}_{i} := \triangle (\tilde{p}\,\tilde{\sigma}_{i}(a_{i})\,\tilde{\sigma}_{i}(u_{i}))$ 
in $\wt{V}_{i}$. 
It follows from (\ref{clairaut}) and (\ref{lem3.4-4}) that 
\begin{equation}\label{lem3.4-6}
\lim_{i \to \infty} \sin \left( \angle(\tilde{p}\,\tilde{\sigma}_{i}(u_{i})\,\tilde{\sigma}_{i}(a_{i})) \right) 
= 
\lim_{i \to \infty} \frac{\nu_{i}}{f(t_{0})} 
= 0
\end{equation}
holds. 
Furthermore, by (\ref{lem3.4-5}), we have 
\begin{equation}\label{lem3.4-7}
\angle(\tilde{p}\,\tilde{\sigma}_{i}(a_{i})\,\tilde{\sigma}_{i}(u_{i})) = \frac{\pi}{2}
\end{equation}
Remark that 
\begin{equation}\label{lem3.4-8}
{\lim_{i \to \infty}\angle(\tilde{\sigma}_{i}(a_{i})\,\tilde{p}\,\tilde{\sigma}_{i}(u_{i})) = 0}_,
\end{equation}
since $\wt{\triangle}_{i} \subset \wt{V}_{i}$ for each $i \in \N$.
By Lemma \ref{lem3.3}, (\ref{lem3.4-7}), (\ref{lem3.4-8}), and the Gauss\,--\,Bonnet Theorem, 
we get 
\begin{equation}\label{lem3.4-9}
{\lim_{i \to \infty} 
\angle(\tilde{p}\,\tilde{\sigma}_{i}(u_{i})\,\tilde{\sigma}_{i}(a_{i})) 
= \frac{\pi}{2}
}_.
\end{equation}
The equation (\ref{lem3.4-9}) contradicts the equation (\ref{lem3.4-6}). 
$\qedd$
\end{proof}

\bigskip

Hebda \cite{He} proved that 
{\em the cut locus $\Cut(x)$ of a point $x$ in a complete Riemannian $2$-manifold is a local tree}, 
that is, for any $y \in \Cut(x)$ and any neighborhood $\cU$ around $y$ in the surface, 
there exists an open neighborhood $\cO \subset \cU$ around $y$ such that any two cut points in $\cO$ 
can be joined by a unique rectifiable Jordan arc in $\cO \cap \Cut(x)$ (see also \cite{ST1} 
for Alexandrov surfaces). 
Here, a {\em Jordan arc} means an arc homeomorphic to the interval $[0, 1]$. 

\begin{theorem}{\bf (Sector Theorem)}\label{thm3.6}
Let $(\wt{M}, \tilde{p})$ be 
a non-compact model surface of revolution satisfying the (VM), 
or the (CH) for some $R_{0} > 0$. 
If $\wt{M}$ admits $c(\wt{M}) > - \infty$, 
then there exists a positive number $\delta_{0} \in (0, \pi)$ 
such that $\wt{V} ( \delta_{0} )$ has no pair of cut points. 
\end{theorem}

\begin{proof} 
Choose any number $R_{1} > R_{0}$, and fix it. 
We will prove that $\wt{V}(\delta (R_{1}))$ has no pair of cut points, 
where $\delta (R_{1}) \in (0, \pi)$ is the number guaranteed in Lemma \ref{lem3.4}. 
Suppose that $\wt{V}(\delta (R_{1}))$ has a pair of cut points $\tilde{x}$ and $\tilde{y}$. 
Let $\tilde{\sigma} : [0, d(\tilde{x}, \tilde{y})] \lra \wt{V}(\delta (R_{1}))$ 
denote a minimal geodesic segment joining $\tilde{x}$ to $\tilde{y}$. 
We may assume that $\tilde{x}$ is conjugate to $\tilde{y}$ along $\tilde{\sigma}$. 
Otherwise, we may find another minimal geodesic segment 
$\tilde{\alpha}$ joining $\tilde{x}$ to $\tilde{y}$. 
Clearly $\tilde{\sigma}$ and $\tilde{\alpha}$ bound a relatively compact domain $\cD$ 
in $\wt{V}(\delta (R_{1}))$. 
Since $\Cut (\tilde{x})$ is a tree, 
we may find an endpoint $\tilde{z} \in \cD$ of $\Cut (\tilde{x})$. 
Then, $\tilde{x}$ is conjugate to $\tilde{z}$ along any minimal geodesic 
segments joining $\tilde{x}$ to $\tilde{z}$. 
By exchanging $\tilde{y}$ and $\tilde{z}$, 
we may assume that $\tilde{x}$ is conjugate to $\tilde{y}$ along $\tilde{\sigma}$. 
From Lemma \ref{lem3.4}, 
the minimal geodesic segment $\tilde{\sigma}$ does not intersect $\ol{B_{R_{1}}(\tilde{p})}$. 
If $G \circ \tilde{\gamma}(R_{1})$ is non-positive, 
then $G \circ \tilde{\gamma}(t) \le 0$ holds on $[R_{1}, \infty)$ where 
$\tilde{\gamma}$ denotes any meridian emanating from $\tilde{p} = \tilde{\gamma} (0)$. 
Hence $G(\tilde{\sigma}(s)) \le 0$ holds for all $s \in [0, d(\tilde{x}, \tilde{y})]$. 
This contradicts the property that $\tilde{x}$ is conjugate to $\tilde{y}$ along $\tilde{\sigma}$. 
If $G \circ \tilde{\gamma}(R_{1})$ is positive and $G \circ \tilde{\gamma}$ is 
non-increasing on $[R_{0}, \infty)$ where 
$\tilde{\gamma}$ denotes any meridian emanating from $\tilde{p} = \tilde{\gamma} (0)$, 
then since $\tilde{y}$ is not a single cut point of $\tilde{x}$, 
we may find a unit speed, rectifiable Jordan arc $\tilde{\xi}(r)$ in $\Cut (\tilde{x})$ 
emanating from $\tilde{y} = \tilde{\xi} (0)$. 
Since $\tilde{\sigma}$ does not intersect $\ol{B_{R_{1}}(\tilde{p})}$, 
there exists a sufficiently small number $\ve_{0} > 0$ such that 
the domain bounded by two minimal geodesic segments 
$\tilde{\tau}$ and $\tilde{\eta}$ joining $\tilde{x}$ to 
$\tilde{\xi} (\ve_{0})$ does not intersect $\ol{B_{R_{1}}(\tilde{p})}$. 
We assume that $\tilde{\tau}$ and $\tilde{\eta}$ are chosen in such a way that 
\[
\angle (\tilde{\eta}'(0), (\partial / \partial t)_{\tilde{x}}) 
< 
\angle (\tilde{\sigma}'(0), (\partial / \partial t)_{\tilde{x}}) 
< 
\angle (\tilde{\tau}'(0), (\partial / \partial t)_{\tilde{x}}). 
\]
Then, we may get a contradiction by repeating the argument 
in the proof of \cite[Lemma 3.1]{GMST}, or \cite[Lemma 3.1]{SiT2}.
In the following, 
we hence state only the sketch of the argument. 
We may prove that $\tilde{\sigma}$ is shorter than $\tilde{\tau}$, 
and that 
\begin{equation}\label{thm3.6-1}
t (\tilde{\sigma}(s)) \ge t (\tilde{\tau}(s))
\end{equation}
for all $s \in [0, d(\tilde{x}, \tilde{y})]$. 
The equation (\ref{thm3.6-1}) implies that  
\begin{equation}\label{thm3.6-2}
G(\tilde{\sigma}(s)) \le G(\tilde{\tau}(s))
\end{equation}
for all $s \in [0, d(\tilde{x}, \tilde{y})]$, 
since $G \circ \tilde{\gamma}$ is non-increasing on $[R_{0}, \infty)$. 
From the Rauch comparison theorem and (\ref{thm3.6-2}), 
the geodesic segment $\tilde{\tau}|_{[0, \, d(\tilde{x}, \tilde{y})]}$ 
has a conjugate point of $\tilde{x}$ along the segment. 
This contradicts the fact that $\tilde{\tau}$ is minimal. 
Therefore, $\wt{V}(\delta (R_{1}))$ has no pair of cut points.
$\qedd$
\end{proof}

\section{A New Type of Toponogov Comparison Theorem}\label{sec:TCT}

In the pure sectional curvature geometry, 
the Toponogov comparison theorem has been 
a very important tool in the investigation of the relationship between the curvature and topology 
of Riemannian manifolds (cf.\,\cite{B}, \cite{CG}, \cite{GS}, \cite{G1}, and so on). 
After the Gromov convergence theorem \cite{G2} and 
the Grove and Peterson finiteness theorem for homotopy 
and diffeomorphism types \cite{GP}, \cite{GPW}, 
it has become very important 
to investigate the topology of Alexandrov spaces
as one of terminal stations 
of the pure sectional curvature geometry (cf.\,\cite{Al}, \cite{BGP}, \cite{P1}, \cite{P2}, 
\cite{O1}, \cite{O2}, \cite{OS}, \cite{ST1}, \cite{SY1}, \cite{SY2}, \cite{Y}, and so on).\par
As stated in Section \ref{sec:int},  
such comparison theorems in the radial curvature geometry 
were proved by making use of model surfaces of revolution 
instead of complete surfaces of constant Gaussian 
curvature (see \cite{A}, \cite{IMS}, and \cite{SiT2}). 
In the radial curvature geometry, 
all geodesic triangles must have the base point as one of vertices. 
Thus, the radial curvature geometry looks more restricted than the pure sectional curvature geometry, 
but this is not the case. 
We should remark that we can construct a model surface of revolution for 
any complete Riemannian manifold with an arbitrary given point 
as a base point (see Model Lemma in Subsection \ref{PML}).\par
Our purpose in this section is to establish 
a new type of Toponogov comparison theorem (Theorem \ref{thm4.8}). 
The key tools of the proof are our Alexandrov convexity (Lemma \ref{lem4.3}) 
and Essential Lemma (Lemma \ref{lem4.6}).  

\subsection{Alexandrov Convexity}

The Alexandrov convexity at a base point was proved in \cite{IMS} 
when comparison surfaces are von Mangoldt surfaces of revolution. 
We first establish our Alexandrov convexity (Lemma \ref{lem4.3}) 
at a base point in a {\em more general situation}.    
For the purpose, 
we need three lemmas (Lemmas \ref{lem3.5}, \ref{newlem4.3.1}, 
and \ref{newlem4.3.2}).

\medskip

Let $(\wt{M}, \tilde{p})$ be an arbitrary non-compact model surface of revolution 
whose metric satisfies (\ref{polar}). 
Then, we have the next lemma, 
which is a very important tool in the investigation of radial curvature geometry. 

\begin{lemma}\label{lem3.5}{\rm(\cite[Lemma 7.3.2]{SST})}
Take a point $\tilde{q} \in \wt{M} \setminus \{\tilde{p} \}$ with $\theta(\tilde{q}) = 0$. 
If two points $\tilde{x}_{1}, \tilde{x}_{2} \in \wt{M}$ satisfy 
$t(\tilde{x}_{1}) = t(\tilde{x}_{2})$ and $0 \le \theta(\tilde{x}_{1}) < \theta(\tilde{x}_{2}) \le \pi$, 
then, 
\[
d(\tilde{q}, \tilde{x}_{1}) < d(\tilde{q}, \tilde{x}_{2})
\]
holds.
\end{lemma}

\medskip

Let $\tilde{\mu}_{\alpha} : [0, \infty) \lra \wt{M}$ denote the meridian defined by $\theta = \alpha$, 
which emanates from $\tilde{p}$. 
It follows from Lemma \ref{lem3.5} that, 
for any $a > 0$, $c > 0$, and $\theta_{0} \in (0, \pi)$, 
\begin{equation}\label{4.3}
{
\liminf_{\theta \downarrow \theta_{0}} 
\frac{
d(\tilde{\mu}_{0}(a), \tilde{\sigma}_{c}(\theta)) - d(\tilde{\mu}_{0}(a), \tilde{\sigma}_{c}(\theta_{0}))
}
{
\theta - \theta_{0}
} 
\ge 0
}_,
\end{equation}
where $\tilde{\sigma}_{c} : [0, 2 \pi) \lra \wt{M}$ denotes the parallel $t = c$, that is, 
$\tilde{\sigma}_{c}(\theta) := \tilde{\mu}_{\theta}(c)$. 
In the next lemma, 
we will prove that the left-hand term in the equation (\ref{4.3}) is strictly positive.

\medskip

\begin{lemma}\label{newlem4.3.1}
For any $a_{0} > 0$, $c_{0} > 0$, and $\theta_{0} \in (0, \pi)$, 
there exist constant numbers $\ve_{1} \in (0, \pi / 2)$ and $\delta > 0$ 
such that 
\begin{equation}\label{newlem4.3.1-eq}
|d(\tilde{\mu}_{0}(a), \tilde{\sigma}_{c}(\theta_{2})) - d(\tilde{\mu}_{0}(a), \tilde{\sigma}_{c}(\theta_{1}))| 
\ge 
\left( 
f(c) \sin \ve_{1} 
\right)
|\theta_{2} - \theta_{1}|
\end{equation}
holds for all $a \in (a_{0} - \delta, a_{0} + \delta)$, 
$c \in (c_{0} - \delta, c_{0} + \delta)$, 
and $\theta_{1}, \theta_{2} \in (\theta_{0} - \delta, \theta_{0} + \delta)$.
\end{lemma}

\begin{proof}
Choose any $a \in (a_{0} - \delta, a_{0} + \delta)$, 
$c \in (c_{0} - \delta, c_{0} + \delta)$,
and $\theta \in (\theta_{0} - \delta, \theta_{0} + \delta)$, 
where $\delta$ is a fixed positive number less than 
\[
{\frac{1}{3} \min 
\left\{
a_{0}, \theta_{0}, c_{0}, \pi - \theta_{0}
\right\}
}_.
\]
Since no minimal geodesic segments joining $\tilde{\mu}_{0}(a)$ to $\tilde{\sigma}_{c}(\theta)$ 
is tangent to the meridian $\tilde{\mu}_{\theta}$, 
there exists a positive constant $\ve_{1} \in (0, \pi / 2)$ such that 
\begin{equation}\label{newlem4.3.1-1}
\Phi (\tilde{\gamma}, \theta) := \angle (\tilde{\sigma}'_{c}(\theta), \tilde{\gamma}'(d(\tilde{\mu}_{0}(a), \tilde{\sigma}_{c}(\theta)))) 
\le \frac{\pi}{2} - \ve_{1}
\end{equation}
holds for all $a \in (a_{0} - \delta, a_{0} + \delta)$, 
$\theta \in (\theta_{0} - \delta, \theta_{0} + \delta)$, 
and all minimal geodesic segments $\tilde{\gamma}$ joining $\tilde{\mu}_{0}(a)$ to $\tilde{\sigma}_{c}(\theta)$. 
Therefore, it follows from \cite[Lemma 2.1]{IT} and (\ref{newlem4.3.1-1}) that, 
for each $\theta_{1} \in (\theta_{0} - \delta, \theta_{0} + \delta)$, 
\begin{align}\label{newlem4.3.1-2}
\liminf_{\theta \downarrow \theta_{1}} 
\frac{
d(\tilde{\mu}_{0}(a), \tilde{\sigma}_{c}(\theta)) - d(\tilde{\mu}_{0}(a), \tilde{\sigma}_{c}(\theta_{1}))
}
{
d(\tilde{\sigma}_{c}(\theta), \tilde{\sigma}_{c}(\theta_{1}))
}
&= 
- \cos \left( \min_{\tilde{\gamma}}\{\pi - \Phi (\tilde{\gamma}, \theta_{1})\} \right)\notag\\[2mm] 
&\ge 
- \cos \left( \frac{\pi}{2}  + \ve_{1}\right)\notag\\[2mm] 
&=
\sin \ve_{1}.
\end{align}
Since 
\[
{
\lim_{\theta \to \theta_{1}}
\frac{d(\tilde{\sigma}_{c}(\theta), \tilde{\sigma}_{c}(\theta_{1}))}{|\theta - \theta_{1}|} 
= f(c)
}_,
\]
we get, by (\ref{newlem4.3.1-2}),
\begin{equation}\label{newlem4.3.1-3}
{
\liminf_{\theta \downarrow \theta_{1}} 
\frac{
d(\tilde{\mu}_{0}(a), \tilde{\sigma}_{c}(\theta)) - d(\tilde{\mu}_{0}(a), \tilde{\sigma}_{c}(\theta_{1}))
}
{
\theta - \theta_{1}
} 
\ge 
f(c) \sin \ve_{1}
}_.
\end{equation}
Since $d(\tilde{\mu}_{0}(a), \tilde{\sigma}_{c}(\theta))$ is Lipschitz with respect to $\theta$, 
it follows from Dini's theorem \cite{D} (cf.\,\cite[Section 2.3]{Hw}, \cite[Theorem 7.29]{WZ}) 
that $d(\tilde{\mu}_{0}(a), \tilde{\sigma}_{c}(\theta))$ is differentiable almost everywhere and 
\[
\int^{\theta_{2}}_{\theta_{1}} 
\frac{\partial}{\partial \theta} 
d(\tilde{\mu}_{0}(a), \tilde{\sigma}_{c}(\theta))\,d \theta 
= 
d(\tilde{\mu}_{0}(a), \tilde{\sigma}_{c}(\theta_{2})) 
- d(\tilde{\mu}_{0}(a), \tilde{\sigma}_{c}(\theta_{1}))
\]
holds 
for all $\theta_{1}, \theta_{2} \in (\theta_{0} - \delta, \theta_{0} + \delta)$. 
Thus, by (\ref{newlem4.3.1-3}), 
\[
|d(\tilde{\mu}_{0}(a), \tilde{\sigma}_{c}(\theta_{2})) - d(\tilde{\mu}_{0}(a), \tilde{\sigma}_{c}(\theta_{1}))| 
\ge 
\left( 
f(c) \sin \ve_{1} 
\right)
|\theta_{2} - \theta_{1}|
\]
holds for all $a \in (a_{0} - \delta, a_{0} + \delta)$, 
$c \in (c_{0} - \delta, c_{0} + \delta)$, 
and $\theta_{1}, \theta_{2} \in (\theta_{0} - \delta, \theta_{0} + \delta)$.
$\qedd$
\end{proof}

\bigskip

It follows from Lemma \ref{lem3.5} that, for any positive numbers $a, b$, and $c$ with 
\[
|a - c| < b < d(\tilde{\mu}_{0}(a), \tilde{\mu}_{\pi}(c)),
\]
there exists a geodesic triangle $\triangle (\tilde{p}\tilde{q}\tilde{r})$ in $\wt{M}$ 
such that 
\[
d(\tilde{p}, \tilde{q}) = a, \quad d(\tilde{q}, \tilde{r}) = b, \quad d(\tilde{r}, \tilde{p}) = c.
\]
Let $\theta (a, b, c)$ denote the angle $\angle (\tilde{r}\tilde{p}\tilde{q})$ of 
the triangle $\triangle (\tilde{p}\tilde{q}\tilde{r})$. 

\medskip

\begin{lemma}\label{newlem4.3.2}
The function $\theta (a, b, c)$ defined on 
\[
T := \left\{(a, b, c) \in \R^{3} \, | \, a, b, c > 0, \, |a - c | < b < d(\tilde{\mu}_{0}(a), \tilde{\mu}_{\pi}(c)) \right\}
\]
is locally Lipschitz.
\end{lemma}

\begin{proof}
Choose any point $(a_{0}, b_{0}, c_{0}) \in T$. 
First, we will prove that 
\begin{equation}\label{newlem4.3.2-1}
|\theta (a_{0} + \varDelta a, b_{0}, c_{0}) - \theta (a_{0}, b_{0}, c_{0})| 
\le \frac{1}{f(c_{0})\sin \ve_{1}}\, |\varDelta a|
\end{equation}
for all $\varDelta a \in \R$ with $|\varDelta a| < \delta$. 
Here, the numbers $\ve_{1}$ and $\delta$ are the constants guaranteed to 
$a_{0}$, $c_{0}$, and $\theta_{0} := \theta (a_{0}, b_{0}, c_{0})$ in Lemma \ref{newlem4.3.1}. 
It follows from Lemma \ref{newlem4.3.1} that 
\begin{equation}\label{newlem4.3.2-2}
|
d(\tilde{\mu}_{0}(a_{0} + \varDelta a), \tilde{\sigma}_{c_{0}}(\theta_{0} + \varDelta_{a} \theta))
- 
d(\tilde{\mu}_{0}(a_{0} + \varDelta a), \tilde{\sigma}_{c_{0}}(\theta_{0}))
| 
\ge 
\left( f(c_{0})\sin \ve_{1} \right) |\varDelta_{a} \theta|
\end{equation}
for all $\varDelta a \in \R$ with $|\varDelta a| < \delta$.
Here, we set 
\[
\varDelta_{a} \theta := \theta (a_{0} + \varDelta a, b_{0}, c_{0}) - \theta_{0}.
\]
It is clear that 
\begin{equation}\label{newlem4.3.2-3}
d(\tilde{\mu}_{0}(a_{0} + \varDelta a), \tilde{\sigma}_{c_{0}}(\theta_{0} + \varDelta_{a} \theta))
= 
d(\tilde{\mu}_{0}(a_{0}), \tilde{\sigma}_{c_{0}}(\theta_{0}))
= b_{0}.
\end{equation}
It follows from the triangle inequality that 
\begin{equation}\label{newlem4.3.2-4}
|
d( \tilde{\mu}_{0}(a_{0}), \tilde{\sigma}_{c_{0}}(\theta_{0}))
- 
d(\tilde{\mu}_{0}(a_{0} + \varDelta a), \tilde{\sigma}_{c_{0}}(\theta_{0}))
| 
\le 
d( \tilde{\mu}_{0}(a_{0}), \tilde{\mu}_{0}(a_{0} + \varDelta a)) 
= 
|\varDelta a|.
\end{equation}
By (\ref{newlem4.3.2-2}), (\ref{newlem4.3.2-3}), and (\ref{newlem4.3.2-4}), 
we get 
\[
|\varDelta_{a} \theta| \le \frac{1}{f(c_{0}) \sin \ve_{1}} |\varDelta a|
\]
for all $\varDelta a \in \R$ with $|\varDelta a| < \delta$. 
Thus, the proof of (\ref{newlem4.3.2-1}) is complete. 
Since 
\[
\theta (a, b, c) = \theta (c, b, a)
\]
for all $(a, b, c) \in T$, 
\begin{equation}\label{newlem4.3.2-5}
|\theta (a_{0}, b_{0}, c_{0} + \varDelta c) - \theta (a_{0}, b_{0}, c_{0})| 
\le \frac{1}{f(c_{0})\sin \ve_{1}}\, |\varDelta c|
\end{equation}
holds for all $\varDelta c \in \R$ with $|\varDelta c| < \delta$. 
We omit the proof of the following equation (\ref{newlem4.3.2-6}), 
since the proof is similar to that of (\ref{newlem4.3.2-1})\,:
\begin{equation}\label{newlem4.3.2-6}
|\theta (a_{0}, b_{0} + \varDelta b, c_{0}) - \theta (a_{0}, b_{0}, c_{0})| 
\le \frac{1}{f(c_{0})\sin \ve_{1}}\, |\varDelta b|
\end{equation}
for all $\varDelta b \in \R$ with $|\varDelta b| < \delta$. 
Therefore, the function $\theta (a, b, c)$ is locally Lipschitz at $(a_{0}, b_{0}, c_{0}) \in T$ 
by (\ref{newlem4.3.2-1}), (\ref{newlem4.3.2-5}), and (\ref{newlem4.3.2-6}). 
$\qedd$
\end{proof}

\medskip

\begin{lemma}\label{lem4.3}{\bf (Alexandrov Convexity)}
Let $(M,p)$ be a complete open Riemannian $n$-manifold $M$ 
whose radial curvature at the base point $p$ is bounded from below by
that of a non-compact model surface of revolution $(\wt{M}, \tilde{p})$. 
For an arbitrary fixed geodesic triangle $\triangle(pxy)$ in $M$, 
let $x,\,y : [0, 1] \lra M$ be its edges 
which are minimal geodesic segments 
joining $p = x(0) = y(0)$ to $x = x(1)$, $y = y(1)$, respectively, 
and are parametrized proportionally to arc-length.\par 
Assume that 
there exists a unique geodesic triangle 
$\wt{\triangle}(p x(t) y(t)) := \triangle(\tilde{p}\tilde{x}(t)\tilde{y}(t))$ in $\wt{M}$ 
up to an isometry corresponding to the triangle $\triangle(px(t)y(t))$ for each $t \in (0, 1)$ 
such that  
\begin{equation}\label{lem4.3-length}
d(\tilde{p},\tilde{x}(t))=d(p, x(t)), \ d(\tilde{p}, \tilde{y}(t))=d(p, y(t)), \ 
d(\tilde{x}(t), \tilde{y}(t))=d(x(t), y(t)),  
\end{equation}
and that 
\begin{equation}\label{lem4.3-angle}
\angle (px(t)y(t)) \ge \angle (\tilde{p}\tilde{x}(t)\tilde{y}(t)), \quad
\angle (py(t)x(t)) \ge \angle (\tilde{p}\tilde{y}(t)\tilde{x}(t)).
\end{equation}
If $\angle (xpy) < \pi$, 
then, the function 
\[
\theta (t) := \angle (\tilde{x}(t)\tilde{p}\tilde{y}(t))
\]
is locally Lipschitz on $(0, 1)$, 
and non-increasing on $(0, 1]$.
\end{lemma}

\begin{proof}
Set 
\[
a := d(p, x), \quad b := d(x, y), \quad c := d(p, y).
\]
Since the edges $x,\,y : [0, 1] \lra M$ of $\triangle(pxy)$ 
are parametrized proportionally to arc-length, respectively, 
we have
\[
a t = d(p, x(t)), \quad c t = d(p, y(t))
\]
for all $t \in [0, 1]$. 
If we define a Lipschitz function on $[0, 1]$ as 
\[
\varphi (t) : = d(x(t), y(t)), 
\]
by the assumption (\ref{lem4.3-length}), 
the function $\theta (t)$ is equal to the function $\theta (at, \varphi (t), c t)$ 
by using the function $\theta(\,\cdot\,, \,\cdot\,,\,\cdot\,)$ defined in Lemma \ref{newlem4.3.2}. 
Hence, $\theta (t)$ is locally Lipschitz by Lemma \ref{newlem4.3.2}. 
By Dini's theorem \cite{D} (cf.\,\cite[Section 2.3]{Hw}, \cite[Theorem 7.29]{WZ}), 
the function $\theta (t)$ is differentiable for almost all $t \in (0, 1)$. 
Thus, we may take any fixed number $t_{0} \in (0, 1)$, 
at which $\theta$ is differentiable. 
By the assumption (\ref{lem4.3-angle}), 
\begin{equation}\label{lem4.3-17}
\angle (px(t_{0})y(t_{0})) \ge \angle (\tilde{p}\tilde{x}(t_{0})\tilde{y}(t_{0})), \quad
\angle (py(t_{0})x(t_{0})) \ge \angle (\tilde{p}\tilde{y}(t_{0})\tilde{x}(t_{0})).
\end{equation}
Let $\tilde{\mu}, \tilde{\eta} : [0, \infty) \lra \wt{M}$ be meridians emanating from $\tilde{p}$ 
and passing through 
$\tilde{x}(t_{0}) = \tilde{\mu} (at_{0})$, $\tilde{y}(t_{0}) = \tilde{\eta} (ct_{0})$, respectively. 
Then, we define a function 
\[
\wt{\psi}(t) := d(\tilde{\mu} (at), \tilde{\eta} (ct))
\]
on $(0, \infty)$. 
Remark that 
\begin{equation}\label{lem4.3-18}
\wt{\psi}(t_{0}) = \varphi (t_{0})
\end{equation}
holds, since 
\[
d(\tilde{\mu} (at_{0}), \tilde{\eta} (ct_{0})) = d(\tilde{x}(t_{0}), \tilde{y}(t_{0})) = d(x(t_{0}), y(t_{0})).
\]
Since $\varphi (t)$ and $\wt{\psi}(t)$ are Lipschitz functions, respectively, 
both functions are differentiable almost everywhere. 
Therefore, we may assume that $\varphi (t)$ and $\wt{\psi}(t)$ are also differentiable at $t = t_{0}$. 
Then, 
\begin{equation}\label{lem4.3-19}
\varphi' (t_{0}) \le \wt{\psi}'(t_{0})
\end{equation}
holds. 
To see (\ref{lem4.3-19}), 
let $q_{0} \in M$ be the midpoint on the edge $x(t_{0})y(t_{0})$ of $\triangle(px(t_{0})y(t_{0}))$, 
and let $\sigma, \tau : [0, \varphi (t_{0})/2] \lra x(t_{0})y(t_{0})$ be minimal geodesic segments 
joining $q_{0} = \sigma (0) = \tau(0)$ to 
$x(t_{0}) = \sigma (\varphi (t_{0})/2)$, $y(t_{0}) = \tau (\varphi (t_{0})/2)$, respectively. 
Then, we see 
\begin{equation}\label{lem4.3-20-1}
\angle (\sigma' (\varphi (t_{0})/2), x'(t_{0})) = \angle (px(t_{0})y(t_{0}))
\end{equation}
and
\begin{equation}\label{lem4.3-20-2}
\angle (\tau' (\varphi (t_{0})/2), y'(t_{0})) = \angle (py(t_{0})x(t_{0})).
\end{equation}
For a sufficiently small fixed number $\ve >0$, 
consider two geodesic variations of $\sigma$ and $\tau$ whose variational curves join $q_{0}$ to $x(t)$, $y(t)$ 
for $t \in (t_{0} - \ve, t_{0} + \ve)$, respectively. 
Since $d(q_{0}, x(t))$ and $d(q_{0}, y(t))$ are differentiable at $t = t_{0}$, 
it follows from the triangle inequality, 
the first variation formula, (\ref{lem4.3-20-1}), and (\ref{lem4.3-20-2}) that we have 
\begin{align}\label{lem4.3-21}
\varphi' (t_{0}) 
&= \lim_{t \downarrow t_{0}} \frac{\varphi(t) - \varphi(t_{0})}{t - t_{0}} \notag\\[2mm]
&\le \lim_{t \downarrow t_{0}} \frac{d(q_{0}, x(t)) + d(q_{0}, y(t)) - d(q_{0}, x(t_{0})) - d(q_{0}, y(t_{0}))}{t - t_{0}} \notag\\[2mm]
&= \lim_{t \downarrow t_{0}} \frac{d(q_{0}, x(t)) - d(q_{0}, x(t_{0}))}{t - t_{0}} 
+ \lim_{t \downarrow t_{0}} \frac{d(q_{0}, y(t)) - d(q_{0}, y(t_{0}))}{t - t_{0}}\notag\\[2mm]
&= \cos \left( \angle (px(t_{0})y(t_{0})) \right) + \cos \left( \angle (py(t_{0})x(t_{0})) \right), 
\end{align}
and  
\begin{align}\label{lem4.3-22}
\varphi' (t_{0}) 
&= \lim_{t \uparrow t_{0}} \frac{\varphi(t) - \varphi(t_{0})}{t - t_{0}} \notag\\[2mm]
&\ge \lim_{t \uparrow t_{0}} \frac{d(q_{0}, x(t)) + d(q_{0}, y(t)) - d(q_{0}, x(t_{0})) - d(q_{0}, y(t_{0}))}{t - t_{0}} \notag\\[2mm]
&= \cos \left( \angle (px(t_{0})y(t_{0})) \right) + \cos \left( \angle (py(t_{0})x(t_{0})) \right). 
\end{align}
Hence, by (\ref{lem4.3-21}) and (\ref{lem4.3-22}), we get 
\begin{equation}\label{lem4.3-23}
\varphi' (t_{0}) = \cos \left( \angle (px(t_{0})y(t_{0})) \right) + \cos \left( \angle (py(t_{0})x(t_{0})) \right).
\end{equation}
By the same way above, we see
\begin{equation}\label{lem4.3-24}
\wt{\psi}'(t_{0}) = \cos \left( \angle (\tilde{p}\tilde{x}(t_{0})\tilde{y}(t_{0})) \right) 
+ \cos \left( \angle (\tilde{p}\tilde{y}(t_{0})\tilde{x}(t_{0})) \right).
\end{equation}
Thus, by (\ref{lem4.3-17}), (\ref{lem4.3-23}), and (\ref{lem4.3-24}), 
we get (\ref{lem4.3-19}), that is, 
\begin{align}
\varphi' (t_{0}) 
&= \cos \left( \angle (px(t_{0})y(t_{0})) \right) + \cos \left( \angle (py(t_{0})x(t_{0})) \right)\notag\\[2mm]
&\le \cos \left( \angle (\tilde{p}\tilde{x}(t_{0})\tilde{y}(t_{0})) \right) 
+ \cos \left( \angle (\tilde{p}\tilde{y}(t_{0})\tilde{x}(t_{0})) \right) 
= \wt{\psi}'(t_{0}).\notag 
\end{align}
Therefore, we get 
\[
\theta'(t_{0}) \le 0
\]
by (\ref{lem4.3-19}). 
Indeed, 
suppose that 
\[
\theta'(t_{0}) > 0.
\]
Then, there exists a constant $\delta > 0$ such that
\begin{equation}\label{lem4.3-25}
\theta (t) - \theta (t_{0}) \ge \delta (t - t_{0})
\end{equation}
holds for all $t > t_{0}$ sufficiently close to $t_{0}$. 
By considering the point $\tilde{\mu}(at)$ as the vertex $\tilde{x}(t)$ of the geodesic triangle 
$\triangle(\tilde{p}\tilde{x}(t)\tilde{y}(t))$ corresponding 
to the geodesic triangle $\triangle(px(t)y(t))$,  
it follows from Lemma \ref{newlem4.3.1} and (\ref{lem4.3-25}) that 
there exists a constant $\ve_{1} \in (0, \pi / 2)$ such that  
\begin{align}\label{lem4.3-26}
\varphi (t) = d(x(t), y(t)) 
&= d(\tilde{\mu}(at), \tilde{y}(t))\notag\\[2mm]
&\ge d(\tilde{\mu}(at), \tilde{\eta}(ct)) + (f(ct) \sin \ve_{1}) (\theta (t) - \theta (t_{0}))\notag\\[2mm]
&\ge \wt{\psi}(t) + (\delta f(ct) \sin \ve_{1}) (t - t_{0})
\end{align}
for all $t > t_{0}$ sufficiently close to $t_{0}$. 
By (\ref{lem4.3-18}) and (\ref{lem4.3-26}), 
\begin{equation}\label{lem4.3-27}
\varphi (t) - \varphi (t_{0}) \ge \wt{\psi}(t) - \wt{\psi}(t_{0}) + (\delta f(ct) \sin \ve_{1}) (t - t_{0}).
\end{equation}
Hence, by (\ref{lem4.3-27}), we get 
\[
\varphi' (t_{0}) \ge \wt{\psi}'(t_{0}) + \delta f(ct) \sin \ve_{1} >  \wt{\psi}'(t_{0}). 
\]
This is a contradiction, 
since $\varphi' (t_{0}) \le \wt{\psi}'(t_{0})$. 
Thus, $\theta'(t) \le 0$ for almost all $t \in (0, 1)$. 
This implies that $\theta (t)$ is non-increasing on $(0, 1]$.
$\qedd$
\end{proof}

\begin{remark}
If the function $\theta(t)$ in Lemma \ref{lem4.3} is not locally Lipschitz, 
then we can not conclude that $\theta$ is non-increasing. 
For example, 
the Cantor--Lebesgue function is increasing on $[0, 1]$ and 
its derivative function is zero almost everywhere. 
By making use of this function, 
we may construct a function which is not increasing, or decreasing, 
but its derivative function is zero almost everywhere (cf.\,\cite{WZ}). 
\end{remark}

\subsection{Toponogov Comparison Theorem}

We are going to show our Toponogov comparison theorem (Theorem \ref{thm4.8}). 
We first introduce the definition of a narrow geodesic triangle 
in an arbitrary complete Riemannian manifold with a base point.

\begin{definition}\label{def4.1}{\rm (\cite[Section 2]{IMS})}
Let $M$ be a complete Riemannian manifold with a base point $p \in M$. 
A geodesic triangle $\triangle(pxy)$ in $M$ is called a {\em narrow geodesic triangle}, 
if 
\[
d(x, y) \ll \min \{ d(p, x),\,d(p, y) \}
\]
and the Fermi coordinates around the edge $px$ 
(respectively $py$) contains the edge $py$ (respectively $px$).   
\end{definition}

From the Rauch--Berger comparison theorem, 
we have the following lemma on a narrow geodesic triangle. 

\begin{lemma}{\rm (\cite[Lemma 2.2]{IMS})}\label{lem4.2}
Let $(M,p)$ be a complete open Riemannian $n$-manifold $M$ 
whose radial curvature at the base point $p$ is bounded from below by
that of a non-compact model surface of revolution $(\wt{M}, \tilde{p})$. 
Then, for every narrow geodesic triangle $\triangle(pxy)$ in $M$,
there exists a geodesic triangle 
$\wt{\triangle} (pxy) :=\triangle(\tilde{p}\tilde{x}\tilde{y})$ in $\wt{M}$ such that
\begin{equation}\label{lem4.2-length}
d(\tilde{p},\tilde{x})=d(p,x), \quad d(\tilde{p},\tilde{y})=d(p,y), \quad d(\tilde{x},\tilde{y})=d(x,y) 
\end{equation}
and that
\begin{equation}\label{lem4.2-angle}
\angle (pxy) \ge \angle (\tilde{p}\tilde{x}\tilde{y}), \quad
\angle (pyx) \ge \angle (\tilde{p}\tilde{y}\tilde{x}).
\end{equation}
\end{lemma}

\medskip

\begin{remark}
Another proof of Lemma \ref{lem4.2} are found in \cite{KT2}. 
Here, from radial curvature geometry's standpoint, 
we establish the Toponogov comparison theorem for open triangles on 
complete manifolds with boundary. 
In particular, it will be clarified in \cite{KT2} that the cut locus of a complete open Riemannian manifold $M$ is not an obstruction at all when we draw a corresponding geodesic triangle in a model surface of revolution for each geodesic triangle in the manifold $M$.
\end{remark} 

Then, we prove the next lemma by using Lemmas \ref{lem3.5} and \ref{lem4.2}.

\begin{lemma}\label{lem4.5}
Let $(M,p)$ be a complete open Riemannian $n$-manifold $M$ 
whose radial curvature at the base point $p$ is bounded from below by
that of a non-compact model surface of revolution $(\wt{M}, \tilde{p})$. 
If the geodesic triangle $\triangle(pxy)$ in $M$ admits a geodesic triangle 
$\wt{\triangle} (pxy) := \triangle (\tilde{p}\tilde{x}\tilde{y})$ in $\wt{V}(\delta_{0})$ 
for some $\delta_{0} \in (0, \pi)$ satisfying 
\begin{equation}\label{lem4.5-length}
d(\tilde{p},\tilde{x})=d(p,x), \quad d(\tilde{p},\tilde{y})=d(p,y), \quad d(\tilde{x},\tilde{y})=d(x,y),  
\end{equation}
then, for any $s \in (0, \ell)$, 
the geodesic triangle $\triangle(px\sigma(s))$ has a geodesic triangle 
$\wt{\triangle}(px\sigma(s))$ in $\wt{V}(\delta_{0})$ satisfying 
(\ref{lem4.5-length}) for $y = \sigma(s)$, 
where $\sigma : [0, \ell] \lra M$ is 
the minimal geodesic segment joining $\sigma(0) = x$ to $\sigma(\ell) = y$. 
\end{lemma}

\begin{proof}
We consider the set $S$ consisting of all $s \in (0, \ell)$ such that, 
for any $r \in (0, s)$, 
there exists a geodesic triangle 
$\wt{\triangle} (px\sigma(r)) := \triangle (\tilde{p}\tilde{x}\tilde{\sigma}(r)) \subset \wt{V}(\delta_{0})$ corresponding to the triangle $\triangle(px\sigma(r)) \subset M$ satisfying 
(\ref{lem4.5-length}) for $y = \sigma(r)$. 
It is clear from Lemma \ref{lem4.2} that $S$ is non-empty set. 
By supposing 
\[
s_{1} := \sup S < \ell, 
\]
we will get a contradiction. 
By definition, 
there exists a geodesic triangle $\wt{\triangle} (px\sigma(s_{1})) \subset \wt{M}$ 
such that 
\begin{equation}\label{lem4.5-1}
\angle(\tilde{x} \tilde{p} \tilde{\sigma}(s_{1})) \ge \delta_{0}
\end{equation}
and (\ref{lem4.5-length}) is valid for $y = \sigma(s_{1})$. 
From Lemma \ref{lem4.2}, 
we may take a subdivision 
\[
0 = s_{0} < s_{1} < \cdots < s_{k} = \ell
\]
of $[0, \ell]$ containing the $s_{1}$ such that all the triangles 
\[
\triangle (p z_{i}z_{i + 1}) 
:= \triangle (p \sigma (s_{i}) \sigma (s_{i + 1})) \subset M
\]
have corresponding triangles $\wt{\triangle} (p z_{i}z_{i + 1}) \subset \wt{M}$
satisfying (\ref{lem4.5-length}) for $x = z_{i}$ and $y = z_{i + 1}$, respectively. 
Here, we set $z_{i} := \sigma(s_{i})$. 
Under this situation, for $i = 1$, 
we draw $\wt{\triangle} (p z_{1} z_{2})$ on $\wt{M}$, 
which is adjacent to $\wt{\triangle} (p z_{0} z_{1})$, 
so as to have a common edge $\tilde{p} \tilde{z}_{1}$. 
Inductively, we draw $\wt{\triangle} (p z_{i} z_{i + 1})$ on $\wt{M}$, 
which is adjacent to $\wt{\triangle} (p z_{i - 1} z_{i})$, 
so as to have a common edge $\tilde{p} \tilde{z}_{i}$. 
Hence, we get a broken geodesic $\tilde{\eta}$ 
emanating from $\tilde{x} = \tilde{z}_{0}$ to $\tilde{z}_{k}$, 
which consists of the opposite sides $\tilde{z}_{i}\tilde{z}_{i + 1}$ to $\tilde{p}$. 
It is trivial that the length $L(\tilde{\eta})$ of $\tilde{\eta}$ is equal to 
\begin{equation}\label{lem4.5-2}
{L(\tilde{\eta}) = \sum_{i = 1}^{k}d(\tilde{z}_{i - 1}, \tilde{z}_{i}) = d(x, y)}_.
\end{equation}
Suppose that the sum of angles 
\[
{\sum_{i = 1}^{k} \angle (\tilde{z}_{i - 1} \tilde{p} \tilde{z}_{i}) \le \pi}_.
\]
Since 
\[
\sum_{i = 1}^{k} \angle (\tilde{z}_{i - 1} \tilde{p} \tilde{z}_{i}) 
> \angle (\tilde{x} \tilde{p} \tilde{\sigma}(s_{1}))
\ge \delta_{0} > \angle (\tilde{x} \tilde{p} \tilde{y})
\]
by (\ref{lem4.5-1}), 
it follows from Lemma \ref{lem3.5} and (\ref{lem4.5-2}) that 
\[
d(\tilde{x}, \tilde{y}) < d(\tilde{x}, \tilde{z}_{k}) \le L(\tilde{\eta}) = d(x, y).
\]
This is a contradiction, 
since $d(\tilde{x}, \tilde{y}) = d(x, y)$. 
Hence, we see that the sum of angles
\[
{\sum_{i = 1}^{k} \angle (\tilde{z}_{i - 1} \tilde{p} \tilde{z}_{i}) > \pi}_.
\]
Without loss of generality, 
we may assume that $\theta(\tilde{x}) = 0$ 
and $\theta \circ \tilde{\eta}$ is increasing. 
Since $\tilde{\eta}$ intersects the meridian $\tilde{\mu}_{\pi}$ defined by $\theta = \pi$, 
we get a unique intersection $\tilde{w}_{0}$. 
The point $\tilde{w}_{0}$ divides $\tilde{\eta}$ 
into two broken geodesics $\tilde{\eta}_{1}$ and $\tilde{\eta}_{2}$. 
Hence, we assume that the subarc emanating from $\tilde{x} = \tilde{z}_{0}$ 
to $\tilde{w}_{0}$ is $\tilde{\eta}_{1}$. 
Since the point $\tilde{\mu}_{\pi}(d(\tilde{p}, \tilde{y}))$ is the nearest point to $\tilde{w}_{0}$ 
on the parallel $t = d(\tilde{p}, \tilde{y})$, 
we have 
\begin{equation}\label{lem4.5-3}
L(\tilde{\eta}) = L(\tilde{\eta}_{1}) + L(\tilde{\eta}_{2}) 
> L(\tilde{\eta}_{1}) + d(\tilde{w}_{0}, \tilde{\mu}_{\pi}(d(\tilde{p}, \tilde{y}))).
\end{equation}
Thus, the broken geodesic $\tilde{\xi}$ consisting of $\tilde{\eta}_{1}$ and the minimal geodesic 
segment joining $\tilde{w}_{0}$ to $\tilde{\mu}_{\pi}(d(\tilde{p}, \tilde{y}))$ is shorter than $\tilde{\eta}$. 
Since 
\[
L(\tilde{\xi}) \ge d(\tilde{x}, \tilde{\mu}_{\pi}(d(\tilde{p}, \tilde{y}))), 
\]
we have, by (\ref{lem4.5-3}), 
\begin{equation}\label{lem4.5-4}
L(\tilde{\eta}) > d(\tilde{x}, \tilde{\mu}_{\pi}(d(\tilde{p}, \tilde{y}))).
\end{equation}
From Lemma \ref{lem3.5}, 
\begin{equation}\label{lem4.5-5}
d(\tilde{x}, \tilde{\mu}_{\pi}(d(\tilde{p}, \tilde{y}))) > d(\tilde{x}, \tilde{y})
\end{equation}
holds, since $\pi > \delta_{0} > \angle (\tilde{x} \tilde{p} \tilde{y})$. 
Hence, by (\ref{lem4.5-4}) and (\ref{lem4.5-5}), we have 
\begin{equation}\label{lem4.5-6}
L(\tilde{\eta}) > d(\tilde{x}, \tilde{y}). 
\end{equation}
Since $L(\tilde{\eta}) = d(x, y) = d(\tilde{x}, \tilde{y})$, 
we get a contradiction from (\ref{lem4.5-6}). 
Therefore, the proof is complete.
$\qedd$
\end{proof}

\begin{lemma}\label{lem4-new.5.31}\ 
Assume that a non-compact model surface of revolution $(\wt{M}, \tilde{p})$ with a metric (\ref{polar}) 
admits a sector $\wt{V}(\delta_{0})$ for some $\delta_{0} \in (0, \pi)$ which has no pair of cut points. 
Let $\triangle (\tilde{p}\tilde{x}_{1}\tilde{y}_{1})$, $\triangle (\tilde{p}\tilde{x}_{2}\tilde{y}_{2})$ be geodesic triangles in $\wt{M}$ satisfying 
\begin{equation}\label{lem4-new.5.31-con4}
d(\tilde{p}, \tilde{y}_{1}) = d(\tilde{p}, \tilde{x}_{2}), 
\end{equation}
and 
\begin{equation}\label{lem4-new.5.31-con3}
\angle (\tilde{p}\tilde{y}_{1}\tilde{x}_{1}) + \angle (\tilde{p}\tilde{x}_{2}\tilde{y}_{2}) \le \pi.  
\end{equation}
If there exists a geodesic triangle 
$\triangle (\tilde{p}\tilde{x}\tilde{y})$ in $\wt{V}(\delta_{0})$ such that 
\begin{equation}\label{lem4-new.5.31-con1}
d(\tilde{p}, \tilde{x}) = d(\tilde{p}, \tilde{x}_{1}), \quad d(\tilde{p}, \tilde{y}) = d(\tilde{p}, \tilde{y}_{2}),
\end{equation}
and that 
\begin{equation}\label{lem4-new.5.31-con2}
d(\tilde{x}, \tilde{y}) = d(\tilde{x}_{1}, \tilde{y}_{1}) + d(\tilde{x}_{2}, \tilde{y}_{2}),  
\end{equation}
then, 
\begin{equation}\label{lem4-new.5.31-con5}
\angle (\tilde{p}\tilde{x}_{1}\tilde{y}_{1}) \ge \angle (\tilde{p}\tilde{x}\tilde{y}), \quad 
\angle (\tilde{p}\tilde{y}_{2}\tilde{x}_{2}) \ge \angle (\tilde{p}\tilde{y}\tilde{x}).
\end{equation}
\end{lemma}

\begin{proof}
From (\ref{lem4-new.5.31-con4}), 
we may draw $\triangle (\tilde{p}\tilde{x}_{2}\tilde{y}_{2})$ on $\wt{M}$, 
which is adjacent to $\triangle (\tilde{p}\tilde{x}_{1}\tilde{y}_{1})$, 
so as to have a common edge $\tilde{p} \tilde{y}_{1} = \tilde{p} \tilde{x}_{2}$, 
i.e., $\tilde{y}_{1} = \tilde{x}_{2}$. 
For the metric (\ref{polar}), we set 
\[
0 := \theta (\tilde{x}_{1}) < \theta (\tilde{y}_{1}) = \theta (\tilde{x}_{2}) <  \theta (\tilde{y}_{2}).
\]
Since the existence of $\triangle (\tilde{p}\tilde{x}\tilde{y}) \subset \wt{V}(\delta_{0})$ is 
unique up to an isometry fixing $\tilde{p}$, 
we may assume, by $d(\tilde{p}, \tilde{x}) = d(\tilde{p}, \tilde{x}_{1})$ of (\ref{lem4-new.5.31-con1}), that
\[
0 = \theta (\tilde{x}) < \theta (\tilde{y}), 
\]
i.e., $\tilde{x} = \tilde{x}_{1}$ and $\theta (\tilde{y}) > 0$. 
Since $\theta (\tilde{y}) < \delta_{0} < \pi$, it is clear from the same argument in the proof of Lemma 
\ref{lem4.5} that 
\begin{equation}\label{lem4-new.10.12-0}
\theta (\tilde{y}_{2}) \le \pi
\end{equation}
holds. By the triangle inequality and (\ref{lem4-new.5.31-con2}), 
\begin{equation}\label{lem4-new.5.31-0}
d(\tilde{x}_{1}, \tilde{y}_{2}) \le d(\tilde{x}_{1}, \tilde{y}_{1}) + d(\tilde{y}_{1}, \tilde{y}_{2}) 
= d(\tilde{x}, \tilde{y}) = d(\tilde{x}_{1}, \tilde{y}).
\end{equation}
Since $\triangle (\tilde{p}\tilde{x}\tilde{y}) \subset \wt{V}(\delta_{0})$, (\ref{lem4-new.10.12-0}), and 
(\ref{lem4-new.5.31-0}), 
it follows from Lemma \ref{lem3.5} that 
$\theta(\tilde{y}_{2}) \le \angle (\tilde{x}_{1}\tilde{p}\tilde{y}) < \delta_{0}$ holds, i.e., 
\begin{equation}\label{lem4-new.5.31-1}
\tilde{y}_{2} \in \theta^{-1} (0, \delta_{0}).
\end{equation}
Let $\wt{\cD}(\delta_{0})$ denote the domain defined by 
\[
\wt{\cD} (\delta_{0}) := \left\{ \tilde{q} \in  \theta^{-1} (0, \delta_{0}) \,|\, 
t(\tilde{q}) < t(\tilde{\sigma} (s))\,\text{\small{whenever}}\,\theta (\tilde{q}) = 
\theta (\tilde{\sigma} (s))\,\text{\small{for some}}\,s > 0 \right\}_.
\]
Here $\tilde{\sigma} : [0, \infty) \lra \wt{M}$ denotes the unit speed geodesic emanating from $\tilde{x}_{1}$ 
passing through $\tilde{y}_{1} = \tilde{\sigma} (d(\tilde{x}_{1}, \tilde{y}_{1}))$. 
Thus, the domain $\wt{\cD} (\delta_{0})$ is bounded by the three subarcs of 
the two meridians $\theta = 0$, $\theta = \delta_{0}$, and $\tilde{\sigma}$.\par
There is nothing to prove, 
if $\angle (\tilde{p}\tilde{y}_{1}\tilde{x}_{1}) + \angle (\tilde{p}\tilde{y}_{1}\tilde{y}_{2}) = \pi$ holds. 
Thus, from (\ref{lem4-new.5.31-con3}), we may assume that 
\begin{equation}\label{lem4-new.5.31-2}
\angle (\tilde{p}\tilde{y}_{1}\tilde{x}_{1}) + \angle (\tilde{p}\tilde{y}_{1}\tilde{y}_{2}) < \pi.
\end{equation}
Hence, by (\ref{lem4-new.5.31-1}) and (\ref{lem4-new.5.31-2}), 
it is clear that $\tilde{y}_{2} \in \wt{\cD} (\delta_{0})$. 
Let $\tilde{c} : [\theta (\tilde{y}_{2}), \delta_{0}] \lra \wt{M}$ denote the subarc of the parallel 
$t = d(\tilde{p}, \tilde{y}_{2})$ that is cut off by two meridians 
$\theta = \theta (\tilde{y}_{2})$ and $\theta = \delta_{0}$. 
Here $\tilde{c}$ is assumed to be parametrized by $\theta$, so that 
$\tilde{c}|_{[\theta (\tilde{y}_{2}), \, \delta_{0})} \subset \wt{V}(\delta_{0})$. 
From Lemma \ref{lem3.5}, $d(\tilde{p}, \tilde{y}) = d(\tilde{p}, \tilde{y}_{2})$ 
of (\ref{lem4-new.5.31-con1}), and (\ref{lem4-new.5.31-0}), it follows that 
\begin{equation}\label{lem4-new.5.31-3}
d(\tilde{x}_{1}, \tilde{c}(\theta (\tilde{y}_{2}))) = d(\tilde{x}_{1}, \tilde{y}_{2}) 
\le d(\tilde{x}_{1}, \tilde{c}(\theta)) \le d(\tilde{x}_{1}, \tilde{c}(\delta_{0}))
\end{equation}
and
\begin{equation}\label{lem4-new.6.20-new}
d(\tilde{x}_{1}, \tilde{y}_{2}) < d(\tilde{x}_{1}, \tilde{y}).
\end{equation}
Moreover, 
it follows from (\ref{lem4-new.5.31-con2}), (\ref{lem4-new.6.20-new}), 
and $\triangle (\tilde{p}\tilde{x}\tilde{y})\subset \wt{V}(\delta_{0})$ that 
\begin{equation}\label{lem4-new.5.31-4}
d(\tilde{x}_{1}, \tilde{c}(\delta_{0})) > d (\tilde{x}_{1}, \tilde{y}) 
= d(\tilde{x}_{1}, \tilde{y}_{1}) + d(\tilde{y}_{1}, \tilde{y}_{2}) > d(\tilde{x}_{1}, \tilde{y}_{2}).
\end{equation}
Therefore, by (\ref{lem4-new.5.31-3}), (\ref{lem4-new.5.31-4}), and the intermediate value theorem, 
there exists a point $\tilde{y}_{3} \in \theta^{-1}(\theta (\tilde{y}_{2}), \delta_{0})$ on 
$\tilde{c}|_{(\theta (\tilde{y}_{2}), \, \delta_{0})}$ such that 
\begin{equation}\label{lem4-new.5.31-5}
d(\tilde{x}_{1}, \tilde{y}_{3}) = d(\tilde{x}_{1}, \tilde{y}_{1}) + d(\tilde{y}_{1}, \tilde{y}_{2}).
\end{equation}
Supposing that $\tilde{y}_{3}$ is not a point in the closure of $\wt{\cD}(\delta_{0})$, 
we will get a contradiction. 
Since $\tilde{y}_{2} \in \wt{\cD}(\delta_{0})$ and $\tilde{y}_{3}$ is not in the closure of $\wt{\cD}(\delta_{0})$, $\tilde{\sigma}$ intersects $\tilde{c}$ 
at a point $\tilde{c}(\theta_{0}) = \tilde{\sigma}(s_{0})$. 
Without loss of generality, 
we may assume that $\tilde{c}|_{[\theta (\tilde{y}_{2}), \, \theta_{0})}$ lies in $\wt{\cD}(\delta_{0})$. 
From Lemma \ref{lem3.5}, it follows that 
\begin{equation}\label{lem4-new.5.31-6}
d(\tilde{y}_{1}, \tilde{y}_{2}) < d(\tilde{y}_{1}, \tilde{c}(\theta_{0})) = d(\tilde{y}_{1}, \tilde{\sigma}(s_{0})).
\end{equation}
Since $\tilde{\sigma}|_{[0, \, s_{0}]}$ is minimal, 
it follows from (\ref{lem4-new.5.31-6}) that 
\begin{equation}\label{lem4-new.5.31-7}
d(\tilde{x}_{1}, \tilde{c}(\theta_{0})) 
= d(\tilde{x}_{1}, \tilde{y}_{1}) + d(\tilde{y}_{1}, \tilde{c}(\theta_{0})) 
> d(\tilde{x}_{1}, \tilde{y}_{1}) + d(\tilde{y}_{1}, \tilde{y}_{2}).
\end{equation}
By using Lemma \ref{lem3.5} again, 
\begin{equation}\label{lem4-new.5.31-8}
d(\tilde{x}_{1}, \tilde{c}(\theta_{0})) < d(\tilde{x}_{1}, \tilde{y}_{3}).
\end{equation}
Combining (\ref{lem4-new.5.31-7}) and (\ref{lem4-new.5.31-8}), we get 
\[
d(\tilde{x}_{1}, \tilde{y}_{3}) > d(\tilde{x}_{1}, \tilde{y}_{1}) + d(\tilde{y}_{1}, \tilde{y}_{2}).
\]
This is a contradiction to (\ref{lem4-new.5.31-5}). 
Thus, $\tilde{y}_{3}$ lies in the closure of $\wt{\cD} (\delta_{0})$. 
Since any subarc of $\tilde{\sigma}$ is minimal when the subarc lies in $\theta^{-1}(0, \delta_{0})$, 
the minimal geodesic joining $\tilde{x}_{1}$ to $\tilde{y}_{3}$ lies 
in the closure of $\wt{\cD} (\delta_{0})$. 
Then, we get a geodesic triangle $\triangle (\tilde{p} \tilde{x}_{1}\tilde{y}_{3}) \subset \wt{V}(\delta_{0})$ 
satisfying (\ref{lem4-new.5.31-con1}) and (\ref{lem4-new.5.31-con2}) for $\tilde{x} = \tilde{x}_{1}$ and 
$\tilde{y} = \tilde{y}_{3}$. 
Since $\wt{V}(\delta_{0})$ has no pair of cut points, 
$\triangle (\tilde{p} \tilde{x}\tilde{y})$ and $\triangle (\tilde{p} \tilde{x}_{1}\tilde{y}_{3})$
are isometric, i.e., $\triangle (\tilde{p} \tilde{x}\tilde{y}) = \triangle (\tilde{p} \tilde{x}_{1}\tilde{y}_{3})$. 
Therefore, $\triangle (\tilde{p} \tilde{x}\tilde{y})$ satisfies 
\[
\angle (\tilde{p}\tilde{x}_{1}\tilde{y}_{1}) \ge \angle (\tilde{p}\tilde{x}\tilde{y})
\] 
of (\ref{lem4-new.5.31-con5}), because the minimal geodesic joining 
$\tilde{x} = \tilde{x}_{1}$ to $\tilde{y} = \tilde{y}_{3}$ lies in the closure of $\wt{\cD} (\delta_{0})$. 
By exchanging $\tilde{x}$ and $\tilde{y}$, and doing $\tilde{x}_{1}$ and $\tilde{y}_{2}$, respectively, 
in the argument above, we may get a geodesic triangle 
$\triangle (\tilde{p} \tilde{x}_{3}\tilde{y}_{2}) \subset \wt{V}(\delta_{0})$ 
satisfying (\ref{lem4-new.5.31-con1}) and (\ref{lem4-new.5.31-con2}) for $\tilde{x} = \tilde{x}_{3}$ and 
$\tilde{y} = \tilde{y}_{2}$, 
and $\angle (\tilde{p}\tilde{y}_{2}\tilde{x}_{2}) \ge \angle (\tilde{p}\tilde{y}_{2}\tilde{x}_{3})$. 
Since both triangles 
$\triangle (\tilde{p} \tilde{x}\tilde{y}), \triangle (\tilde{p} \tilde{x}_{3}\tilde{y}_{2}) \subset \wt{V}(\delta_{0})$ 
are isometric, $\angle (\tilde{p}\tilde{y}\tilde{x}) = \angle (\tilde{p}\tilde{y}_{2}\tilde{x}_{3})$ holds, i.e., 
we get 
\[
\angle (\tilde{p}\tilde{y}_{2}\tilde{x}_{2}) \ge \angle (\tilde{p}\tilde{y}\tilde{x})
\] 
of (\ref{lem4-new.5.31-con5}). 
Therefore, we have proved our lemma. 
$\qedd$
\end{proof}

By Lemmas \ref{lem4.2}, \ref{lem4.5}, and \ref{lem4-new.5.31}, 
we have the next lemma. 

\begin{lemma}\label{lem4.6}{\bf (Essential Lemma for Toponogov Comparison Theorem)}\ \par
Let $(M,p)$ be a complete open Riemannian $n$-manifold $M$ 
whose radial curvature at the base point $p$ is bounded from below by
that of a non-compact model surface of revolution $(\wt{M}, \tilde{p})$. 
Assume that $\wt{M}$ admits a sector $\wt{V}(\delta_{0})$ 
for some $\delta_{0} \in (0, \pi)$ which has no pair of cut points. 
If a geodesic triangle $\triangle(pxy)$ in $M$ admits a geodesic triangle 
$\wt{\triangle} (pxy) := \triangle (\tilde{p}\tilde{x}\tilde{y})$ in $\wt{V}(\delta_{0})$ satisfying 
\begin{equation}\label{lem4.6-length}
d(\tilde{p},\tilde{x})=d(p,x), \quad d(\tilde{p},\tilde{y})=d(p,y), \quad d(\tilde{x},\tilde{y})=d(x,y),  
\end{equation}
then 
\begin{equation}\label{lem4.6-angle}
\angle (pxy) \ge \angle (\tilde{p}\tilde{x}\tilde{y}), \quad
\angle (pyx) \ge \angle (\tilde{p}\tilde{y}\tilde{x}).
\end{equation}
\end{lemma}

\begin{proof}
Let $\triangle (pxy)$ denote any geodesic triangle in $M$ which admits a geodesic triangle 
$\wt{\triangle} (pxy)$ in $\wt{V}(\delta_{0})$ satisfying (\ref{lem4.6-length}). 
We denote by $\sigma : [0, \ell] \lra M$ the edge of $\triangle (pxy)$ opposite to $p$. 
Let $S$ be the set of all $r \in (0, \ell]$ such that 
there exists a geodesic triangle 
$\wt{\triangle} (px\sigma(r)) := \triangle (\tilde{p}\tilde{x}\tilde{\sigma}(r)) \subset \wt{V}(\delta_{0})$ corresponding to the triangle $\triangle(px\sigma(r)) \subset M$ satisfying 
(\ref{lem4.6-length}) and (\ref{lem4.6-angle}) for $y = \sigma(r)$. 
It is clear from Lemma \ref{lem4.2} that $S$ is non-empty. 
Since there is nothing to prove in the case where $\sup S = \ell$, 
we suppose that 
\[
s_{1} := \sup S < \ell.
\]
By Lemma \ref{lem4.5}, 
there exists a geodesic triangle $\wt{\triangle} (px\sigma(s_{1})) \subset \wt{V}(\delta_{0})$ corresponding to the triangle $\triangle(px\sigma(s_{1})) \subset M$ satisfying 
(\ref{lem4.6-length}) and (\ref{lem4.6-angle}) for $y = \sigma(s_{1})$. 
For a sufficiently small $\ell - s_{1} > \ve > 0$, it follows from Lemma \ref{lem4.2} that 
there exists a geodesic triangle $\wt{\triangle} (p \sigma(s_{1}) \sigma(s_{1} + \ve))$ 
corresponding to the triangle $\triangle(p \sigma(s_{1}) \sigma(s_{1} + \ve)) \subset M$ 
satisfying (\ref{lem4.6-length}) and (\ref{lem4.6-angle}) 
for $x = \sigma(s_{1})$ and $y = \sigma(s_{1} + \ve)$. 
Thus, two geodesic triangles $\wt{\triangle} (px\sigma(s_{1})) $ and 
$\wt{\triangle} (p \sigma(s_{1}) \sigma(s_{1} + \ve))$ satisfy 
(\ref{lem4-new.5.31-con4}) and (\ref{lem4-new.5.31-con3}) (in Lemma \ref{lem4-new.5.31}) 
for $\tilde{x}_{1} = \tilde{x}$, $\tilde{y}_{1} = \tilde{\sigma}(s_{1})$, 
$\tilde{x}_{2} = \tilde{\sigma}(s_{1})$, and $\tilde{y}_{2} = \tilde{\sigma}(s_{1} + \ve)$. 
On the other hand, by Lemma \ref{lem4.5} again, 
we get a geodesic triangle $\triangle (\tilde{x} \tilde{p} \hat{y}) \subset \wt{V}(\delta_{0})$ 
satisfying (\ref{lem4-new.5.31-con1}) and (\ref{lem4-new.5.31-con2}) (in Lemma \ref{lem4-new.5.31}) 
for $\tilde{y} = \hat{y}$, $\tilde{x}_{1} = \tilde{x}$, $\tilde{y}_{1} = \tilde{\sigma}(s_{1})$, 
$\tilde{x}_{2} = \tilde{\sigma}(s_{1})$, and $\tilde{y}_{2} = \tilde{\sigma}(s_{1} + \ve)$. 
Thus, it follows from Lemma \ref{lem4-new.5.31} that 
$s_{1} + \ve \in S$. 
This contradicts the fact that $s_{1}$ is the supremum of $S$.
$\qedd$
\end{proof}

Then, by Lemmas \ref{lem4.3} and \ref{lem4.6}, 
we have a new type of Toponogov comparison theorem, 
which is the main theorem in this section.

\medskip

\begin{theorem}\label{thm4.8}
{\bf (A New Type of Toponogov  Comparison Theorem)}\ \par
Let $(M,p)$ be a complete open Riemannian $n$-manifold $M$ 
whose radial curvature at the base point $p$ is bounded from below by
that of a non-compact model surface of revolution $(\wt{M}, \tilde{p})$. 
Assume that $\wt{M}$ admits a sector $\wt{V}(\delta_{0})$ 
for some $\delta_{0} \in (0, \pi)$ which has no pair of cut points. 
Then, for every geodesic triangle $\triangle(pxy)$ in $M$ with $\angle (xpy) < \delta_{0}$,
there exists a geodesic triangle 
$\wt{\triangle} (pxy) :=\triangle(\tilde{p}\tilde{x}\tilde{y})$ in $\wt{V}(\delta_{0})$ such that
\begin{equation}\label{thm4.8-length}
d(\tilde{p},\tilde{x})=d(p,x), \quad d(\tilde{p},\tilde{y})=d(p,y), \quad d(\tilde{x},\tilde{y})=d(x,y) 
\end{equation}
and that
\begin{equation}\label{thm4.8-angle}
\angle (xpy) \ge \angle (\tilde{x}\tilde{p}\tilde{y}), \quad  
\angle (pxy) \ge \angle (\tilde{p}\tilde{x}\tilde{y}), \quad
\angle (pyx) \ge \angle (\tilde{p}\tilde{y}\tilde{x}). 
\end{equation}
\end{theorem}

\medskip

\begin{proof}
Let $W$ be the set of all $t \in (0, 1)$ such that 
there exists a geodesic triangle 
$\wt{\triangle} (px(t)y(t)) := \triangle (\tilde{p}\tilde{x}(t)\tilde{y}(t)) \subset \wt{V}(\delta_{0})$ corresponding to the triangle $\triangle(px(t)y(t)) \subset M$ satisfying 
(\ref{thm4.8-length}) for $x = x(t)$ and $y = y(t)$. 
Here $x(t)$ and $y(t)$ denote the minimal geodesic segments introduced in Lemma \ref{lem4.3}, 
respectively. 
It is clear that $W$ is open. 
From the Rauch comparison theorem, 
there exists a constant $\ve_{0} > 0$ such that, 
for each $t \in (0, \ve_{0})$, 
there exists a geodesic triangle 
$\wt{\triangle}(px(t)y(t)) \subset \wt{V}(\delta_{0})$ 
corresponding to $\triangle(px(t)y(t)) \subset M$ satisfying 
(\ref{thm4.8-length}) for $x = x(t)$ and $y = y(t)$ and 
\[
\theta (t) := \angle (\tilde{x}(t) \tilde{p} \tilde{y}(t)) \le \angle (x(t) p y(t)) = \angle (x p y) < \delta_{0}.
\]
Hence we get 
\[
(0, \ve_{0}) \subset W.
\]
Let $(0, t_{0}) \subset W$ denote the connected component of $W$ containing $(0, \ve_{0})$. 
From Lemmas \ref{lem4.3} and \ref{lem4.6}, 
it follows that $\theta (t)$ is non-increasing on $(0, t_{0})$. 
Thus, the geodesic triangle $\triangle(px(t_{0})y(t_{0}))$ has a corresponding triangle 
$\wt{\triangle} (px(t_{0})y(t_{0})) \subset \wt{V}(\delta_{0})$ satisfying (\ref{thm4.8-length}) 
for $x = x(t_{0})$ and $y = y(t_{0})$. This implies that $t_{0} \in W$ if $t_{0} < 1$. 
Therefore, $W = (0, 1)$ and the proof is complete. 
$\qedd$
\end{proof}

\begin{remark}
We refer to \cite{O3} for a generalization of the Toponogov comparison theorems 
to the Finsler geometry, and its applications.
\end{remark}

\section{Application of Main Theorem}\label{sec:MOC}

In this section, we first give the proof of Model Lemma (Lemma \ref{lem5.1}). 
After recalling some differential inequality (Lemma \ref{DIL}), 
we give the proof of our answer (Theorem \ref{thm5.3}) to Milnor's open conjecture. 
Finally, we prove a corollary (Corollary \ref{cor5.4}) to Main Theorem.

\subsection{Proof of Model Lemma}\label{PML}

Let $M$ be an arbitrary complete open Riemannian $n$-manifold. 
Fix {\bf any} point $p \in M$. Then, we have 

\begin{lemma}\label{lem5.1}
There exists a locally Lipschitz function $G(t)$ on $[0, \infty)$ 
such that the radial curvature of $M$ at $p$ is bounded from below by 
that of the non-compact model surface of revolution with radial curvature function $G(t)$. 
\end{lemma}

\begin{proof}
We set 
\[
\Sph^{n - 1}_{p} := \{ v \in T_{p} M \ | \ \| v \| = 1 \}.
\] 
Let $\gamma_{v} : [0, \rho(v)] \lra M$ be a minimal geodesic emanating from 
$p = \gamma_{v}(0)$ such that $v = \gamma'_{v}(0) \in \Sph^{n - 1}_{p}$, 
where we set 
\[
\rho(v) := \sup \{ t > 0 \ | \ d(p, \gamma_{v}(t)) = t \}.
\] 
Then, take an orthonormal basis 
\[
\{e_{1}, e_{2}, \cdots, e_{n - 1} \} := \{e_{1}(v), e_{2}(v), \cdots, e_{n - 1}(v) \}
\]
of the hyperplane in $T_{p}M$ orthogonal to $e_{n} := v$. 
Then, we denote by $E_{i}(t\,; v)$, $i = 1, 2, \cdots, n$, parallel vector fields along $\gamma_{v}$ 
such that $E_{i}(0\,; v) = e_{i}$. 
Moreover, let $\sigma_{t}$ be a $2$-dimensional linear subspace of $T_{\gamma_{v}(t)}M$ 
spanned by $\gamma'_{v}(t)$ and a tangent vector $w_{t}$ to $M$ at $\gamma_{v}(t)$ 
defined by
\[
w_{t} := \sum_{i = 1}^{n - 1}a_{i}E_{i}(t\,; v)
\]
with 
\[
\sum_{i = 1}^{n - 1}a_{i}^{2} = 1
\]
for $a_{1}, a_{2}, \cdots, a_{n -1} \in \R$. 
Thus, the radial sectional curvature $K_M(\sigma_{t})$ at $p$ of $M$ is given by 
\begin{equation}\label{5.1-1}
{K_M(\sigma_{t}) 
= \left\langle R(E_{n}(t\,; v), w_{t})E_{n}(t\,; v), w_{t} \right\rangle}_,
\end{equation}
where $R$ denotes the Riemannian curvature tensor of $M$, which is a multi-linear map, defined by 
\[
R(X, Y)Z := \nabla_{Y}\nabla_{X}Z -  \nabla_{X}\nabla_{Y}Z + \nabla_{[X, Y]}Z
\]
for smooth vector fields $X, Y, Z$ over $M$. 
Now, we set 
\[
\mbox{\boldmath $a$} 
:=
\left(
\begin{array}{c}
a_{1}\\
a_{2}\\
\vdots\\
a_{n -1}   
\end{array}
\right)
\in \Sph^{n - 2}(1),
\]
where $\Sph^{n - 2}(1) \subset \R^{n - 1}$ is a standard unit sphere, 
and set 
\[
{R_{ij}(t, v) := \left\langle R(E_{n}(t\,; v), E_{i}(t\,; v))E_{n}(t\,; v), E_{i}(t\,; v) \right\rangle}_. 
\]
Remark that $R_{ij}(t, v)$ is a $C^{\infty}$-function.
Then, by (\ref{5.1-1}), we see that 
\begin{equation}\label{5.1-2}
K_M(\sigma_{t})
= \sum_{i,\,j = 1}^{n - 1}a_{i}a_{j}R_{ij}(t, v) 
= {}^{t}\!\mbox{\boldmath $a$} R(t\,; v) \mbox{\boldmath $a$}_,
\end{equation}
where $R(t\,; v) := \left( R_{ij}(t, v) \right)$ 
is a symmetric $(n -1)\!\times\!(n -1)$-matrix. 
Furthermore, we set 
\[
F_{0}(t, v) := 
\min \left\{
{}^{t}\!\mbox{\boldmath $a$} R(t\,; v) \mbox{\boldmath $a$} \ | \ 
\mbox{\boldmath $a$} \in \Sph^{n - 2}(1) 
\right\}
\]
for all $(t, v) \in [0, \infty) \times \Sph^{n - 1}_{p}$ with $t \le \rho(v)$. 
Thus, by (\ref{5.1-2}), we get 
\begin{equation}\label{5.1-3}
K_M(\sigma_{t}) \ge F_{0}(t, v)
\end{equation}
for all $(t, v) \in [0, \infty) \times \Sph^{n - 1}_{p}$ with $t \le \rho(v)$. 
Remark that $F_{0}(t, v)$ is locally Lipschitz on 
$[0, \infty) \times \Sph^{n - 1}_{p}$ with $t \le \rho(v)$, 
since $R_{ij}(t, v)$ is a $C^{\infty}$-function.\par
We define 
\[
F(t, v) := F_{0}(\rho_{t}(v), v)
\]
on $[0, \infty) \times \Sph^{n - 1}_{p}$, where we set $\rho_{t}(v) := \min \{\rho (v), t\}$. 
Since $F(t, v)$ is locally Lipschitz on $[0, \infty) \times \Sph^{n - 1}_{p}$ with $t \le \rho(v)$, 
for any $b > 0$, 
there exits a constant $C_{0}$
such that
\begin{equation}\label{5.1-3.1}
|F(t_{1}, v) - F(t_{2}, v)| \le C_{0} |\rho_{t_{1}}(v) - \rho_{t_{2}}(v)|
\end{equation}
for all $t_{1}, t_{2} \in [0, b]$ and all $v \in \Sph^{n - 1}_{p}$. 
Since it is clear that 
\[
|\rho_{t_{1}}(v) - \rho_{t_{2}}(v)| \le |t_{1} - t_{2}|
\]
for all $t_{1}, t_{2} \in [0, b]$ and all $v \in \Sph^{n - 1}_{p}$, 
it follows from (\ref{5.1-3.1}) that 
\begin{equation}\label{5.1-3.1.1}
|F(t_{1}, v) - F(t_{2}, v)| \le C_{0} |t_{1} - t_{2}|
\end{equation}
holds for all $t_{1}, t_{2} \in [0, b]$ and all $v \in \Sph^{n - 1}_{p}$. 
Then, we get a locally Lipschitz function 
\begin{equation}\label{5.1-4}
G(t) := \min \left\{ F(t, v) \ | \ v \in \Sph^{n - 1}_{p} \right\}
\end{equation}
on $[0, \infty)$. 
Indeed, take any $t_{1}, t_{2} \in [0, b]$. 
Without loss of generality, 
we may assume $G(t_{1}) \ge G(t_{2})$. 
By (\ref{5.1-3.1.1}), 
\begin{align}\label{5.1-4.1}
G(t_{1}) - G(t_{2}) 
&= G(t_{1}) - F(t_{2}, v_{2})  \notag\\[2mm]
&\le F(t_{1}, v_{2}) - F(t_{2}, v_{2})  
\le C_{0}|t_{1} - t_{2}|. \notag
\end{align}
Here, $v_{2} \in \Sph^{n - 1}_{p}$ is a point satisfying 
$G(t_{2}) = F(t_{2}, v_{2})$. 
Thus, we have proved that $G(t)$ is locally Lipschitz on $[0, \infty)$. 
Therefore, it follows from (\ref{5.1-3}) and (\ref{5.1-4}) that 
\[
K_M(\sigma_{t}) \ge G(t)
\]
holds.
$\qedd$
\end{proof}

\subsection{Proof of Partial Answer to Milnor's Open Conjecture}\label{PAMOC}

Before starting the proof of our answer to Milnor's open conjecture, 
we will recall the following differential inequality (compare \cite[Lemma 2.1]{Z}). 

\begin{lemma}\label{DIL}
Let $\phi(t)$ be a $C^{1}$-function on $[0, \infty)$, 
and $\lambda (t)$ a continuous function on $[0, \infty)$. 
If 
\[
\phi'(t) \le \lambda (t) \phi(t)
\]
holds on $[0, \infty)$, 
then we have 
\[
\phi (t) \le e^{\Lambda (t)} \phi (0).
\]
Here, we set 
\[
{\Lambda (t) := \int_{0}^{t}\lambda (s)ds}_.
\]
\end{lemma}

\begin{proof}
Since 
$0 \ge \phi'(t) - \lambda (t) \phi(t)$ on $[0, \infty)$ 
by the assumption on this lemma, 
we see 
\begin{equation}\label{DIL-1}
0 \ge e^{- \Lambda (t)}\phi'(t) - e^{- \Lambda (t)}\lambda (t) \phi(t) 
= \frac{d}{dt} \left( e^{- \Lambda (t)}\phi (t) \right)
\end{equation}
for all $t \in [0, \infty)$. 
Thus, by this (\ref{DIL-1}), we have 
\begin{align}\label{DIL-2}
0  = \int_{0}^{t} 0\,ds 
&\ge \int_{0}^{t} \frac{d}{ds} \left( e^{- \Lambda (s)}\phi (s) \right) ds \notag\\[2mm]
&= e^{- \Lambda (t)}\phi (t) - e^{- \Lambda (0)}\phi (0) \notag\\[2mm]
&= e^{- \Lambda (t)}\phi (t) - \phi (0)
\end{align}
for all $t \in [0, \infty)$. 
Therefore, by this (\ref{DIL-2}), we get 
\[
e^{- \Lambda (t)}\phi (t) \le \phi (0), 
\]
that is, 
$
\phi (t) \le e^{\Lambda (t)}\phi (0)
$ 
holds on $[0, \infty)$. 
$\qedd$
\end{proof}

\bigskip

Now, let $G(t)$ be the Lipschitz function in Model Lemma (Lemma \ref{lem5.1}), and set 
\[
G^{*}(t) := \min \left\{ 0, G(t) \right\}_.
\]
Consider a non-compact model surface of revolution $(M^{*}, p^{*})$ 
with its metric 
\begin{equation}\label{model-metric}
g^{*} = dt^2 +  m(t)^2d \theta^2, \quad (t,\theta) \in (0,\infty) \times \Sph_{p^{*}}^1
\end{equation}
satisfying the differential equation
$m''(t) + G^{*} (t) m(t) = 0$
with initial conditions $m(0) = 0$ and $m'(0) = 1$. 
Remark that the metric (\ref{model-metric}) is not always differentiable around 
the base point $p^{*} \in M^{*}$. 

\begin{theorem}\label{thm5.3}
Let $M$ be a complete open Riemannian $n$-manifold, $p \in M$ {\bf any} fixed point, 
and $(M^{*}, p^{*})$ a comparison model surface of revolution, constructed from $(M, p)$, 
with its metric (\ref{model-metric}). 
If $G^{*} (t)$ satisfies 
\begin{equation}\label{thm5.3-0}
\int^{\infty}_{0} \left( - t \cdot G^{*} (t) \right) dt < \infty, 
\end{equation}
then the total curvature $c(M^{*})$ is finite. 
In particular, then $M$ has a finitely generated fundamental group. 
\end{theorem}

\begin{proof}
Since $G^{*}(t) \le 0$ on $[0, \infty)$, 
we have 
\[
m'' (t) = - G^{*} (t)m(t) \ge 0
\]
on $[0, \infty)$, that is, 
$m'(t)$ is increasing on $[0, \infty)$. 
Since $m'(0) = 1$, 
\[
m' (t) \ge m'(0) = 1
\]
holds for all $t \in [0, \infty)$. 
In particular, $m(t)$ is non-negative for all $t \ge 0$, 
for $m(t)$ is increasing on $[0, \infty)$ and $m(0) = 0$. 
We will first prove the total curvature $c(M^{*})$ of $(M^{*}, p^{*})$ is finite. 
To show this fact, 
we define 
\[
\phi (t) := m'(t)
\]
on $[0, \infty)$. 
Then, since $m(0) = 0$ and $m'(t)$ is increasing on $[0, \infty)$, 
we see that 
\begin{equation}\label{5.2-2}
\frac{m(t)}{t} = \frac{1}{t} \int_{0}^{t} m'(s) ds \le \frac{m'(t)}{t} \int_{0}^{t} ds = \phi (t)
\end{equation}
on $(0, \infty)$. 
Thus, by (\ref{5.2-2}), we have 
\begin{align}
\phi' (t) = m'' (t) 
&= - G^{*}(t)m(t)\notag \\[2mm]
&= \left( - t\,G^{*}(t) \right) \cdot \frac{m(t)}{t}\notag \\[2mm]
&\le - t\,G^{*}(t) \phi (t)\notag \\[2mm] 
&= \lambda (t)\phi (t)\notag
\end{align}
for all $t \in (0, \infty)$. 
Here, we set 
$\lambda (t) := - t\,G^{*}(t)\,(\ge 0)$.
In particular, 
\[
\phi' (t) \le \lambda (t)\phi (t)
\]
holds on $[0, \infty)$, 
since $G^{*} (t)$ is continuous on $[0, \infty)$.
Then, it follows from Lemma \ref{DIL} and $m'(0) = 1$ that we have 
\begin{equation}\label{5.2-3}
\phi (t) \le e^{\Lambda (t)} \phi (0) = e^{\Lambda (t)} m' (0) = e^{\Lambda (t)}
\end{equation}
on $[0, \infty)$. 
Here, we set 
\[
{\Lambda (t) := \int_{0}^{t}\lambda (s)ds}_.
\]
Since 
\[
0 \le \Lambda (\infty) 
= \int_{0}^{\infty}\lambda (s)ds
= \int^{\infty}_{0} \left( - s \cdot G^{*}(s) \right) ds < \infty
\]
by the assumption (\ref{thm5.3-0}),
we see, by (\ref{5.2-3}), 
\begin{equation}\label{5.2-4}
{\lim_{t \to \infty} m'(t) 
= \lim_{t \to \infty} \phi (t) 
\le \lim_{t \to \infty} e^{\Lambda (t)} 
= e^{\Lambda (\infty)} < \infty
}_.
\end{equation}
By (\ref{5.2-4}), we get 
\[
{c(M^{*}) = 2 \pi \int_{0}^{\infty} \left( - m''(t) \right) dt = 2 \pi \left( 1 - \lim_{t \to \infty} m'(t) \right) > - \infty}_.
\]
Thus, $c(M^{*})$ is finite.\par
Therefore, 
it follows from Sector Theorem and Main Theorem that $M$ has finite topological type, 
that is, $M$ is homeomorphic to the interior of a compact manifold with boundary. 
In particular, $M$ therefore has a finitely generated fundamental group. 
$\qedd$
\end{proof}

\bigskip

Finally, we prove a corollary to 
our Main Theorem, which is another partial answer to Milnor's open conjecture.

\medskip

Let $K_{M}$ be the sectional curvature of an arbitrary complete Riemannian $n$-manifold $M$. 
Furthermore, we define a (real) number $K_{M}(q)$ as follows
\[
{K_{M} (q) := \min_{\sigma \subset T_{q}M} K_{M}(\sigma)}_.
\]
Here, $\sigma \subset T_{q}M$ is a $2$-dimensional linear space, 
and $K_{M}(\sigma)$ is the sectional curvature of $\sigma$ at $q \in M$.

\begin{corollary}\label{cor5.4}
Let $M$ be an arbitrary complete open Riemannian $n$-manifold, $p \in M$ {\bf any} fixed point, 
and $(M^{*}, p^{*})$ a comparison model surface of revolution, constructed from $(M, p)$, 
with its metric (\ref{model-metric}). 
If the sectional curvature $K_{M}$ of $M$ satisfies 
\begin{equation}\label{cor5.4-0}
\liminf_{t \to \infty} t^{2 + \alpha} \min_{q \in B_{t}(p)} K_{M}(q) > \cN
\end{equation}
for some numbers $\alpha > 0$ and $\cN \le 0$, 
then $M$ has a finitely generated fundamental group. 
\end{corollary}

\begin{proof}
By the assumption (\ref{cor5.4-0}), we have 
\begin{equation}\label{cor5.4-1}
{t^{2 + \alpha} \min_{q \in B_{t}(p)} K_{M}(q) > C \cdot \cN}_,
\end{equation}
for some $C \ge 1$ and all $t > 1$. 
Then, it follows from the construction of $G(t)$ (see the proof of Lemma \ref{lem5.1}) 
and (\ref{cor5.4-1}) that 
\begin{equation}\label{cor5.4-2}
{G(t) \ge \min_{q \in B_{t}(p)} K_{M}(q) > \frac{C \cdot \cN}{t^{2 + \alpha}}}_.
\end{equation}
holds for all $t > 1$. 
Since $G^{*}(t) \le 0$ on $[0, \infty)$, we see, by (\ref{cor5.4-2}), 
\begin{equation}\label{cor5.4-3}
{0 \ge \int_{1}^{\infty} t \cdot G^{*}(t)\,dt 
\ge \int_{1}^{\infty} \frac{C \cdot \cN}{t^{1 + \alpha}} dt 
= \left[ -\,\frac{C \cdot\cN}{\alpha t^{\alpha}} \right]_{1}^{\infty} 
= \frac{C \cdot\cN}{\alpha} 
> - \infty}_.
\end{equation}
Hence, we get, by (\ref{cor5.4-3}), 
\[
{0 \le \int_{0}^{\infty} -\,t \cdot G^{*}(t)\,dt < \infty}_.
\]
Therefore, by Sector Theorem and Main Theorem, 
$M$ has a finitely generated fundamental group. 
$\qedd$
\end{proof}

\begin{remark}
Under the assumption in Theorem \ref{thm5.3}, or Corollary \ref{cor5.4}, 
it follows from \cite[Theorem 0.1]{MNO} 
that $(M, p)$ admits the asymptotic cone via rescaling argument, i.e., 
the pointed Gromov--Hausdorff limit space of $((1/t)M, p)$ exists as $t \to \infty$, and 
the space is, naturally, isometric to a Euclidean cone (see \cite[Definition 3.14]{G2} for a definition of the pointed Gromov--Hausdorff convergence). However, one should notice again that 
our models in Theorem \ref{thm5.3} and Corollary \ref{cor5.4} have been constructed from 
any complete open Riemannian manifold with an arbitrary given point as a base point, and that 
the metrics (\ref{model-metric}) in Theorem \ref{thm5.3} and Corollary \ref{cor5.4} are not always differentiable around their base points. In particular, 
our Main Theorem has a wider class of metrics than those described in \cite[Theorem 0.1]{MNO}. 
\end{remark}

\bigskip
The next example shows that 
non-negative radial curvature {\em does not always mean} non-negative sectional curvature. 

\begin{example}\label{exa1.2}
Let $M$ be a 2-sphere of 
revolution with a Riemannian metric 
\begin{equation}\label{1.2}
h := dr^2 + m(r)^2d \theta^{2}, \quad (r,\theta) \in (0,2a) \times {\Sph_{p}^1}_. 
\end{equation}
and pair of poles $p, q$, i.e., the surface $(M, h)$ satisfies that 
\begin{enumerate}[{\rm ({TS--}1)}]
\item
$(M, h)$ is symmetric with respect to the reflection fixing $r = a$, where $2a = d(p, q)$, 
\item
the radial curvature function $G \circ \gamma : [0, 2a] \lra \R$ of $M$ is monotonic along a meridian $\gamma$ emanating from 
$p = \gamma (0)$ to the point on $r = a$.
\end{enumerate} 
Remark that $(M, h)$ does not always have positive radial curvature function $G \circ \gamma (t)$\,: 
For example, the model surface of revolution generated by the $(x, z)$-plane curve $(m(r), 0, z(r))$ 
satisfies (TS--1) and (TS--2), where 
\[
m(r) := \frac{\sqrt{3}}{10} 
\left(
9 \sin \frac{\sqrt{3}}{9} r + 7 \sin \frac{\sqrt{3}}{3} r
\right)_, 
\qquad 
{z(r) := \int_{0}^{r} \sqrt{1 - m'(t)^{2}} dt}_.
\]
In particular, we then see $G(\gamma(3\sqrt{3}\pi / 2)) = -1$ (see \cite{SiT2}).\par
Thus, without loss of generality, by setting $2a := \pi$, 
we may assume that (\ref{1.2}) is the geodesic polar coordinates 
around the north pole $(0, 0, 1)$ of the unit sphere $\Sph^{2}(1)$ 
in $3$-dimensional Euclidean space $\R^3$, 
and that the radial curvature function $G \circ \gamma (t)$ of $(M, h)$ 
is {\em negative} at a point on $(0, \pi)$. 
Now, we will introduce a new Riemannian metric $g$ 
on $3$-dimensional Euclidean space $(\R^{3}, g_{0})$. 
Outside of the unit ball $B_{1}(o) \subset \R^{3}$ centered at the origin $o \in \R^{3}$, 
we define $g$ to be 
\[
g := dt^2 + f(t)^2 h,
\]
where $f : (0, \infty) \lra (0, \infty)$ is a smooth function, 
and the function $t$ denotes the Euclidean distance function from $o \in \R^{3}$. 
We set $x_{1} := t$, $x_{2} := r$, and $x_{3} := \theta$, 
and denote by $\sigma_{ij}$ a $2$-dimensional linear plane spanned 
by $\partial / \partial x_{i}$ and $\partial / \partial x_{j}$, $i \not= j$. 
Then, the sectional curvatures $K(\sigma_{ij})$ of the planes $\sigma_{ij}$ at each points on 
$\R^{3} \setminus B_{1}(o)$
are 
\begin{equation}\label{1.3}
K(\sigma_{12}) = K(\sigma_{13}) = -\,\frac{f''(t)}{f(t)}
\end{equation}
and 
\begin{equation}\label{1.4}
K(\sigma_{23}) = \frac{1}{f(t)^2}\left(-\,\frac{m''(r)}{m(r)} - f'(t)^{2} \right)_.
\end{equation}
Consider a smooth family $\{ h_{t} \}_{t \ge 0}$ of Riemannian metrics such that 
$h_{t}$ is the standard metric on $\Sph^{2}(1)$ for small $t$ and $h_{t} = h$ for $t \ge 1$. 
Then, $(t, r, \theta)$ are the geodesic polar coordinates 
around $o \in \R^{3}$ for $(\R^{3}, g_{0})$, 
and the Riemannian metric $g_{t} := dt^2 + f(t)^2 h_{t}$ on $\R^{3}$ 
equals $g_{0}$ for small $t$ if $f(t) = t$, 
and equals $g$ on $\R^{3} \setminus B_{1}(o)$. 
By the definition of $g_{t}$, each $t$-curve on $(\R^{3}, g_{t})$ is a ray emanating from $o \in \R^{3}$. 
Therefore, 
it follows from (\ref{1.3}) and (\ref{1.4}) that 
the radial curvature of $(\R^{3}, g)$ is non-negative on $\R^{3} \setminus B_{1}(o)$, if $f''(t) \le 0$. 
In particular, $(\R^{3}, g)$ admits non-negative Ricci curvature at divergent points.
\end{example}

\medskip

From Example \ref{exa1.2}, 
we may conclude Milnor's open conjecture does not follow from our Main Theorem, and, 
at the same time, {\em our Main Theorem does not follow from Milnor's open conjecture}. 
Moreover, we find that the radial curvature geometry deals with a geometry different
from the geometry of a global lower bound on Ricci curvature, 
but with a sufficiently large geometry.

\medskip

\begin{center}
Kei KONDO $\cdot$ Minoru TANAKA 

\bigskip
Department of Mathematics\\
Tokai University\\
Hiratsuka City, Kanagawa Pref.\\ 
259\,--\,1292 Japan

\bigskip
$\bullet$\,our e-mail addresses\,$\bullet$

\bigskip 
\textit{e-mail of Kondo} \,:

\medskip
{\tt keikondo@keyaki.cc.u-tokai.ac.jp}

\medskip
\textit{e-mail of Tanaka}\,:

\medskip
{\tt m-tanaka@sm.u-tokai.ac.jp}
\end{center}

\end{document}